\documentclass[12pt]{article}
\usepackage{amsmath,amsfonts,amssymb}
\usepackage{alltt}
\usepackage{mathtools}  

\usepackage{amsmath,amsfonts,ifthen,fullpage,enumerate,float}  

\usepackage{amsmath,amsfonts,amssymb}
\usepackage{graphicx}        
\def\spam{\mathop{\rm span}\nolimits}

\def\diag{\mathop{\rm diag}\nolimits}
\def\GL{\mathop{\rm GL}\nolimits}
\def\trace{\mathop{\rm trace}\nolimits}

\def\rank{\mathop{\rm rank}\nolimits}

\def\pmat#1{\begin{pmatrix}#1\end{pmatrix}}

\def\question#1{{\bf Question: }#1}
\def\question#1{}

\def\cF{{\cal F}}

\def\CC{\mathbb{C}}

\def\ZZ{\mathbb{Z}}

\def\Cd{\C^d}

\def\C{\mathbb{C}}

\newtheorem{theorem}{Theorem}[section]

\newtheorem{lemma}{Lemma}[section]
\newtheorem{example}{Example}[section]

\newtheorem{proposition}{Proposition}[section]
\newtheorem{definition}{Definition}[section]

\newenvironment{proof}{{\noindent \it
Proof.}}{\hfill$\Box$\medskip}
\newif\ifdraft\def\draft{\drafttrue\hoffset=.8truecm\showlabeltrue
\def\comment##1{{\bf comment: ##1}}
\headline={\sevenrm \hfill \ifx\filenamed\undefined\jobname\else\filenamed\fi%
(.tex) (as of \ifx\updated\undefined???\else\updated\fi)
 \TeX'ed at {\hour\time\divide\hour by 60{}%
\minutes\hour\multiply\minutes by 60{}%
\advance\time by -\minutes
\the\hour:\ifnum\time<10{}0\fi\the\time\  on \today\hfill}}
}

\def\inpro#1{\langle#1\rangle}
\def\ip#1{\langle\kern-.28em\langle#1\rangle\kern-.28em\rangle_\nu}

\def\cU{{\cal U}}

\def\norm#1{\Vert#1\Vert}
\def\openR{{{\rm I}\kern-.16em {\rm R}}}

\let\ga\alpha

\let\gb\beta

\let\gd\delta

\let\gl\lambda

\let\gL\Lambda

\let\gs\sigma

\let\ga\alpha

\let\gb\beta
\let\gd\delta
\let\gs\sigma
\def\inpro#1{\langle#1\rangle}

\def\Iff{\hskip1em\Longleftrightarrow\hskip1em}
\def\Implies{\hskip1em\Longrightarrow\hskip1em}

\def\formeq{\the\sectionno.\the\equationno}  
\def\elabel#1/#2/#3/{\global\advance\equationno by 1 %
\ifx#1\empty\else\emember#1%
\ifshowlabel\marginal{\string#1}\fi\fi%
\ifmmode\eqno{#3(\formeq#2)}\else#3\formeq#2\fi} 

\def\makeblanksquare#1#2{
\dimen0=#1pt\advance\dimen0 by -#2pt
      \vrule height#1pt width#2pt depth0pt\kern-#2pt
      \vrule height#1pt width#1pt depth-\dimen0 \kern-#1pt
      \vrule height#2pt width#1pt depth0pt \kern-#2pt
      \vrule height#1pt width#2pt depth0pt
}

\title{\bf 
The Fourier transform \\
of a projective group frame
}

\author{Shayne Waldron\\
 \\
Department of Mathematics \\ University of Auckland\\
Private
Bag 92019, Auckland, New Zealand\\
e--mail: waldron@math.auckland.ac.nz}
\date{\today}

\begin{document}

\maketitle 

\begin{abstract}
Many tight frames of interest are constructed via their Gramian matrix
(which determines the frame up to unitary equivalence).
Given such a Gramian, it can be determined whether or not the tight
frame is projective group frame, i.e., is the projective orbit of some 
group $G$ (which may not be unique). 
On the other hand, there is complete description of the projective group
frames in terms of the irreducible projective representations of $G$.
Here we consider the inverse problem of taking the Gramian of a projective
group frame for a group $G$, and identifying the cocycle 
and constructing the frame explicitly as 
the projective group orbit of a vector $v$ (decomposed in terms of
the irreducibles). The key idea is to recognise that the Gramian
is a group matrix given by a vector $f\in\CC^G$, and 
to take the Fourier transform of $f$ to obtain
the 
components of $v$ as orthogonal projections.

This requires the development of a theory of
group matrices and the Fourier transform for projective representations.
Of particular interest, we give a block diagonalisation of (projective) group matrices.
This leads to a unique Fourier decomposition of the group matrices, 
and a further fine-scale decomposition into low rank group matrices.
\end{abstract}

\bigskip
\vfill

\noindent {\bf Key Words:}
$(G,\alpha)$-frame,
Group matrix,
Gramian matrix,
projective group frame,
twisted group frame,
central group frame,
Fourier transform,
character theory,
Schur multiplier,
projective representation.

\bigskip
\noindent {\bf AMS (MOS) Subject Classifications:} 
primary
20C15, \ifdraft	Ordinary representations and characters \else\fi
20C25, \ifdraft	Projective representations and multipliers \else\fi
42C15, \ifdraft General harmonic expansions, frames  \else\fi
43A32, \ifdraft	Other transforms and operators of Fourier type \else\fi
65T50, \ifdraft	Discrete and fast Fourier transforms \else\fi

\quad
secondary
05B30, \ifdraft (Other designs, configurations) \else\fi
42C40, \ifdraft (Wavelets) \else\fi
43A30, \ifdraft	Fourier and Fourier-Stieltjes transforms on nonabelian groups and on semigroups, etc. \else\fi
94A12. \ifdraft (Signal theory [characterization, reconstruction, etc.]) \else\fi

\vskip .5 truecm
\hrule
\newpage

\section{Introduction}

Let $G$ be a finite abstract group. A map
$\rho:G\to\GL(\Cd)$ is said to be 
a {\bf projective representation} of $G$
of dimension $d=d_\rho$ with {\bf cocycle} (or {\bf multiplier})
$\ga:G\times G\to\CC$ if 
\begin{equation}
\label{prorepeqn}
\rho(g)\rho(h) = \ga(g,h) \rho(gh), \qquad \forall g,h\in G.
\end{equation}
Two projective representations $\rho,\ga$ and $\tilde\rho,\tilde\ga$
are {\bf equivalent} if there is a map $c:G\to\CC$ 
with $\tilde\rho = c T\rho T^{-1}$. 
The theory of projective representations follows that of {\it ordinary}
representations (when $\ga=1$), and will be introduced as needed.
The key point for now, is that for certain groups there are representations 
which are not ordinary representations (the possible $\ga$ are indexed
by the Schur multiplier group).

For a nonzero vector $v\in\Cd$ and a projective unitary representation
$\rho,\ga$, we define a {\bf projective group frame} 
or {\bf $(G,\ga)$-frame} to be the 
projective $G$-orbit of $v$, i.e.,
the sequence of vectors
$$ (\phi_g)=(gv)=(\rho(g)v)_{g\in G}, \qquad g v:=\rho(g)v. $$
By (\ref{prorepeqn}),
this satisfies
\begin{equation}
\label{Galphasymmeq}
g\phi_h =\rho(g)\rho(h)v = \ga(g,h)\phi_{gh},
\end{equation}
and if the representation is unitary, i.e., $\rho(g)^*=\rho(g)^{-1}$, then
\begin{equation}
\label{projGramian}
\inpro{\phi_h,\phi_g} = \inpro{\rho(h)v,\rho(g)v} 
= \inpro{\rho(g)^{-1}\rho(h)v,v}
= \inpro{ {\rho(g^{-1}h)v\over\ga(g,g^{-1}h)},v}.
\end{equation}
These 
generalise {\it group frames} (also known as {\it $G$-frames}),
which are the case when $\ga=1$.
A sequence of vectors $(v_j)$ in $\Cd$ is a ({\bf normalised}) {\bf tight frame} for $\Cd$ 
if it satisfies 
\begin{equation}
\label{normalisedtfdefn}
f = \sum_{j=1}^n \inpro{f,v_j}v_j, \qquad\forall f\in\Cd.
\end{equation}
The term {\it Parseval frame} is also commonly used. 
Normalised tight $G$-frames, which are natural generalisations of orthonormal bases, 
have numerous applications \cite{CK13}, \cite{BS11}, \cite{FJMP15} \cite{TH17},
and their structure is
now well understood 
\cite{W18}. 

On the other hand, tight $(G,\ga)$-frames are
of great interest also \cite{HL00}, \cite{GH03}. These
are sometimes referred to as {\it projective frame representations}
or {\it group-like systems}. As an example, {\it SICs }
(sets of $d^2$ equiangular lines
in $\Cd$) 
come as a tight projective group frame for $\ZZ_d\times\ZZ_d$ \cite{ACFW17},
as do {\it MUBs} ({\it mutually unbiased bases}) \cite{WF89}, \cite{GR09}, 
and many real and complex {\it spherical $t$-designs} \cite{GKMS16}, \cite{W17}.

Until recently, such tight frames have been studied by considering them
as $G$-frames for a larger group for which $\rho(G)$ contains scalar matrices,
and by accounting for the scalar multiples of a given vector. 
One such approach is to consider a {\it canonical abstract error group}
(the enlarged group)
with {\it index group} $G$ \cite{CW17}.
Recent calculations, most notably  \cite{CH18} (which uses the term
{\it twisted} group frame), suggest that they can be viewed as 
projective group frames, with the theory of group frames
following with little extra work (a twist if you like). 
Our contribution to this burgeoning
theory of projective group frames is to study them via their 
{\bf Gramian} matrix
$V^*V\in\CC^{G\times G}$, $V=[\phi_g]_{g\in G}$, 
which has $(g,h)$-entry given by (\ref{projGramian}).
The key points and results are the following:
\begin{itemize}
\item
The Gramian determines a projective group frame up to unitary equivalence.
If the representation $\rho$ is unitary (as must be the case for a tight frame), then the
Gramian is 
a $(G,\ga)$-matrix 
(see definition \ref{Galphaframedef}).
\item
Each $(G,\ga)$-matrix is determined by a $\nu\in\CC^G$, and they form a
$C^*$-algebra $M_{(G,\ga)}$.
\item
The tight $(G,\ga)$-frames correspond to the orthogonal projections
in $M_{(G,\ga)}$.
\item Each tight $(G,\ga)$-frame $(gv)_{g\in G}$ can be decomposed 
$v=\oplus_j v_j$, $v_j\in V_j$, where the action of $\rho$ on 
$V_j$ is irreducible.
\item Given the Gramian of tight projective group frame for $G$,
i.e., a $(G,\ga)$-matrix $P$ which is an orthogonal projection, 
we find a vector $v=\oplus_j v_j$ and representations $\rho_j \in R$,
where $R$ is a complete set of irreducible representations for $\ga$, 
such that the projective group frame $(\sum_j\rho_j(v_j))_{g\in G}$ 
has Gramian $P$. In other words, we give a concrete construction of a
projective group frame (known only from its Gramian) 
in terms of the irreducibles involved
(Theorem \ref{constructthm}). 
\item The construction above relies on a simultaneous unitary block 
diagonalisation of the $(G,\ga)$-matrices (Theorem \ref{Galphamatdiag}), 
related to
an appropriately defined Fourier transform for projective representations
of a finite group. The blocks preserve the spectral structure of the 
$(G,\ga)$-matrix. In particular, the tight $(G,\ga)$-frames correspond
to the case when all of the blocks are orthogonal projections.
This characterisation of the tight $(G,\ga)$-frames, gives the usual 
characterisation (in terms of the decomposition of $v$ into irreducibles),
and gives a natural description of the {\it central $(G,\ga)$-frames}
(as those where the blocks are $0$ or $I$).
The decomposition of $(G,\ga)$-matrices given by our Fourier transform
is of independent interest, e.g., it allows the determinant of a general
$(G,\ga)$-matrix to be factored. There is also a fine-scale decomposition 
into low rank $(G,\ga)$-matrices, which is not unique. 
\end{itemize}

We now proceed, following the above outline.

\section{The Schur multiplier}

We first consider the functions $\ga:G\times G\to\CC$ which can be
cocycles of a projective representation of $G$. 
Given (\ref{prorepeqn}), multiplying out
$\{\rho(x)\rho(y)\}\rho(z) = \rho(x)\{\rho(y) \rho(z)\}$, 
leads to 
the following multiplication rule for cocycles
\begin{equation}
\label{alpharule}
\ga(x,y) \ga(xy,z) = \ga(x,yz) \ga(y,z).
\end{equation}
Every $\ga$ satisfying (\ref{alpharule}) does come from a projective
representation. Indeed, with $(e_g)_{g\in G}$ being the standard basis vectors
for $\CC^G$, we can define $\rho:G\to\GL(\CC^{G})$ by
\begin{equation}
\label{alpharegrep}
\rho(g) e_h :=\ga(g,h)e_{gh},
\end{equation}
and use (\ref{alpharule}) to verify that it is such a representation:
\begin{align*} \rho(g)\rho(h)e_k 
&= \rho(g)\ga(h,k)e_{hk} 
= \ga(h,k)\ga(g,hk)e_{ghk}  \cr
&= \ga(g,h)\ga(gh,k)e_{ghk} 
= \ga(g,h)\rho(gh)e_{k} .
\end{align*}
\vfil\eject

Henceforth a map $\ga:G\times G\to\CC$ satisfying 
(\ref{alpharule}) will be called a {\bf cocycle} 
(more properly a {\bf $2$-cocycle}) or {\bf multiplier} of $G$. 
The set of cocycles is an abelian group under pointwise multiplication,
which is commonly denoted by $Z^2(G,\CC^\times)$. 
If projective representations $\rho,\ga$ and $\tilde\rho,\tilde\ga$ are
are {\bf equivalent}, i.e., $\tilde\rho = c T\rho T^{-1}$,
where $c:G\to\CC$, then 
$$ \tilde\rho(g)\tilde\rho(h)
=c_g T\rho(g)T^{-1}c_h T\rho(h) T^{-1} 
=c_g c_h T \ga(g,h) \rho(gh) T^{-1}
= { c_g c_h\over c_{gh}} \ga(g,h) \tilde\rho(gh), $$
so that
$$ \tilde \ga(g,h) = \gb(g,h) \ga(g,h), \qquad
\gb(g,h) :=  { c_g c_h\over c_{gh}}. $$
The function $\gb$ above is a cocycle, called a {\bf coboundary}
(or {$2$-coboundary}), since
$$\gb(x,y)\gb(xy,z) 
= {c_xc_y\over c_{xy}} {c_{xy} c_{z}\over c_{xyz}}
= {c_xc_yc_z\over c_{xyz}}
= {c_xc_{yz}\over c_{xyz}} {c_yc_z\over c_{yz}}
= \gb(x,yz)\gb(y,z). $$
The coboundaries form a subgroup $B^2(G,\CC^\times)$ of $Z^2(G,\CC^\times)$. 
The quotient group 
$$ M(G)=H^2(G,\CC^\times):=Z^2(G,\CC^\times)/B^2(G,\CC^\times)$$
is called the {\bf Schur multiplier} (or {\bf second homology group}
$H_2(G,\ZZ)$ of $G$). The Schur multiplier is a finite abelian group 
whose exponent divides the order of $G$. If $G$ has a nontrivial cyclic 
Sylow $p$-subgroup, then $p$ does not divide $|M(G)|$. There do exist 
finite groups with nontrivial Schur multipliers, and hence 
projective representations which are not ordinary. The first few cases
of nontrivial Schur multipliers are for certain groups of order $4,8,9,12$. 

Every projective representation $\rho,\ga$ is equivalent to one 
$\tilde\rho,\tilde\ga$, where the $\tilde\rho$ is unitary. 
For a unitary representation, 
$|\det(\rho(g))|=1$, and so by
taking determinants of (\ref{prorepeqn}), we have
\begin{equation}
\label{unitarycocycle}
|\ga(g,h)| = 1, \qquad g,h\in G.
\end{equation}
Consequently, a cocycle satisfying (\ref{unitarycocycle}) 
is said to be a {\bf unitary}
cocycle. For the purpose of defining the Schur multiplier, one 
can suppose that all of the cocycles are unitary, and satisfy the 
normalisation condition that $\ga(1,1)=1$.

We now list some 
properties of cocycles that we will often use.
Since 
$$\rho(1)^2=\ga(1,1)\rho(1),
\qquad
\rho(g)\rho(g^{-1}) = \ga(g,g^{-1}) \rho(1), $$
we have
$\rho(1)=\ga(1,1)I$, and
\begin{equation}
\label{rhoinverse}
\rho(g)^{-1} = {\rho(g^{-1})\over\ga(g,g^{-1})\ga(1,1)}.
\end{equation}
Similarly, 
$\rho(g)\rho(1)=\ga(g,1)\rho(g)$ and $\rho(1)\rho(g)=\ga(1,g)\rho(g)$, 
give
\begin{equation}
\label{alphag1}
\ga(1,g)=\ga(g,1)=\ga(1,1), \quad
\ga(g,g^{-1})=\ga(g^{-1},g), \qquad\forall g\in G.
\end{equation}

%


\section{The $C^*$-algebra of $(G,\ga)$-matrices}

Motivated by the formula (\ref{projGramian}) for the Gramian 
of a $(G,\ga)$-frame, we have: 

\begin{definition}
\label{Galphaframedef}
We say that $A=[a_{g,h}]_{g,h\in G}\in\CC^{G\times G}$ is a {\bf $(G,\ga)$-matrix} if
\begin{equation}
\label{Malphanuformula}
a_{g,h} = M_\ga(\nu)_{g,h}:= {\nu(g^{-1}h)\over\ga(g,g^{-1}h)},
\end{equation}
for some $\nu:G\to\CC$. We denote the set
of $(G,\ga)$-matrices by $M_{(G,\ga)}$. 
\end{definition}
Other variations 
are discussed in \S\ref{otherGgamats},
e.g., \cite{CH18} consider the matrices
 $M(\nu)=M_{1/\ga}(\nu)$.

Given a $(G,\ga)$-matrix $A\in M_{(G,\ga)}$, the cocycle $\ga$
(or part of it) can be recovered via
$$ {\ga(g,h)\over\ga(1,1)} 
= {a_{1,h}\over a_{g,gh}} 
= {\nu(h)\over\ga(1,h)}{\ga(g,h)\over\nu(h)}, \qquad \nu(h)\ne0, $$
$$ {\ga(g,h)\over\ga(1,1)} 
= {a_{gh,g}\over a_{h,1}}
= {\nu(h^{-1})\over\ga(gh,h^{-1})}{\ga(h,h^{-1})\over\nu(h^{-1})},
\qquad \nu(h^{-1})\ne0. $$

From now on, we let $n=|G|$, and refer to matrices indexed by the
elements of $G$ as
$n\times n$ matrices or $G\times G$ matrices.

\begin{lemma}
The $(G,\ga)$-matrices are an $n$--dimensional subspace of the
$n\times n$ matrices which is closed under matrix multiplication, i.e.,
$$ M_\ga( a\nu+b\mu) = aM_\ga(\nu)+b M_\ga(\mu), \qquad a,b\in\CC,$$
\begin{equation}
\label{Malphaconv}
M_\ga(\nu)M_\ga(\mu)=M_\ga(\nu*_\ga \mu),
\qquad (\nu*_\ga \mu)(g) := \sum_{t\in G} 
{ \nu(g t^{-1})\mu(t) \over \ga(gt^{-1},t)}.
\end{equation}
When $\ga$ is unitary, i.e., $|\ga|=1$, then they
are also closed under the Hermitian transpose 
\begin{equation}
\label{Malphahermtrans}
(M_\ga(\nu))^* = M_\ga(\nu^{*,\ga}), \qquad
\nu^{*,\ga}(a) := \overline{\nu(a^{-1})}\ga(a,a^{-1})\ga(1,1).
\end{equation}
\end{lemma}

\begin{proof}
Since
$\ga(g,g^{-1}a)\ga(a,a^{-1}h) = \ga(g,g^{-1}h) \ga(g^{-1}a,a^{-1}h)$,
we have
\begin{align*}
\bigl(M_\ga(\nu)M_\ga(\mu)\bigr)_{g,h}
&= \sum_{a\in G} M_\ga(\nu)_{g,a} M_\ga(\mu)_{a,h}
= \sum_{a\in G} {\nu(g^{-1}a)\over\ga(g,g^{-1}a)} {\mu(a^{-1}h)\over\ga(a,a^{-1}h)} \cr
&= {1\over \ga(g,g^{-1}h)} \sum_{a\in G} 
{\nu(g^{-1}h (a^{-1}h)^{-1})\mu(a^{-1}h)\over\ga(g^{-1}h (a^{-1}h)^{-1},a^{-1}h)}
= {(\nu*_\ga \mu)(g^{-1}h)\over \ga(g,g^{-1}h)}.
\end{align*}
Since
$\ga(h,h^{-1}g)\ga(g,g^{-1}h) = \ga(h,1)\ga(h^{-1}g,g^{-1}h),$
we have
\begin{align*}
(M_\ga(\nu)^*)_{g,h} & = \overline{M_\ga(\nu)_{h,g}}
=\overline{\nu(h^{-1}g)}\ga(h,h^{-1}g) \cr
&= { \overline{\nu((g^{-1}h)^{-1})}\ga(g^{-1}h,(g^{-1}h)^{-1})\ga(1,1)
\over\ga(g,g^{-1}h)} =M_\ga(\nu^{*,\ga})_{g,h}.
\end{align*}
We will call $\nu *_\ga \mu$ the {\bf $\ga$-convolution} (of $\nu$ and $\mu$).
\end{proof}

\begin{example} A calculation gives
$$ e_g *_\ga e_h = {e_{gh}\over\ga(g,h)}, $$
so that
$$ 
M_\ga(e_g) M_\ga(e_h) = 
{M_\ga( e_{gh})\over \ga(g,h) } . $$
In particular, taking $g=h=1$, gives $I=\ga(1,1)M_\ga(e_1)$, 
so the identity matrix is a $(G,\ga)$-matrix, hence
(by Cayley-Hamilton) the inverse of a nonsingular $(G,\ga)$-matrix 
is a $(G,\ga)$-matrix. Further, for $\ga$ unitary, the limit formulas
$$ A^+
=\lim_{\gd\to0^+} (A^*A+\gd I)^{-1} A^*
=\lim_{\gd\to0^+} A^*(AA^*+\gd I)^{-1} $$
for the pseudoinverse, shows that $M_{(G,\ga)}$ is 
closed under the pseudoinverse.
\end{example}

\begin{example}
For $\ga$ unitary, the matrix $P=M_\ga(\nu)$ is an orthogonal projection, 
i.e., satisfies $P^2=P$, $P^*=P$, if and only if
$\nu*_\ga \nu =\nu$ and $\nu^{*,\ga} = \nu$.
\end{example}

\begin{example}
\label{starrhojk}
For $\rho$ (and hence $\ga$) unitary, we have $\rho_{jk}^{*,\ga}=\rho_{kj}$, 
$1\le j,k\le d_\rho$, i.e.,
$$ M_\ga(\rho_{jk})^* = M_\ga(\rho_{kj}). $$
This follows from (\ref{rhoinverse}),
by taking the $(k,j)$-entry of 
$ \rho(g) = \ga(g,g^{-1})\ga(1,1) \rho(g^{-1})^* $.
\end{example}




\section{The twisted group algebra}

We now describe how the basic theory of 
representations extends to the projective case. 
Let $\rho:G\to\GL(V)$ be a projective representation of $G$ for a cocycle $\ga$
on $V$.
This gives a ``projective group action'' of $G$ on the $\CC$-vector space $V$
\begin{equation}
\label{modulemult}
g\cdot v = gv := \rho(g)v, \qquad\forall v\in V.
\end{equation}
This is not a {\it group action} in the usual sense, since
$$ g\cdot (h\cdot v) = \ga(g,h) (gh\cdot v). $$
Nevertheless, we will talk about $G$-invariant subspaces of $V$, etc.
As in the ordinary case, we say that a 
representation $\rho,\ga$ on $V$ is {\bf irreducible} if
the only $G$-invariant subspaces of $V$ are $0$ and $V$, i.e., for
any nonzero vector $v$ the $G$-orbit $(g\cdot v)_{g\in G}$ spans $V$.

Let $\CC G$ denote the set of formal $\CC$-linear combinations of the
elements of $G$. This becomes are ring, which we denote by $(\CC G)_\ga$,
under the multiplication
$$ g\cdot_\ga h := \ga(g,h)gh, $$
extended linearly. We note that the cocycle multiplication
rule (\ref{alpharule}) is equivalent to the associativity of the 
multiplication, since
\begin{align*}
(g\cdot_\ga h)\cdot_\ga k
&=\ga(g,h) (gh\cdot_\ga k) =\ga(g,h)\ga(gh,k) ghk \cr
g\cdot_\ga (h\cdot_\ga k)
&= g\cdot_\ga (\ga(h,k) hk)
= \ga(h,k) \ga(g,hk) ghk.
\end{align*}
Moreover, $(\CC G)_\ga$ is an algebra, which generalises the 
group algebra $\CC G$ (the case $\ga=1$), and is called
the {\bf $\ga$-twisted group algebra} (over $\CC$).
The vector space $V$ becomes an $(\CC G)_\ga$-module
under the operation $(\CC G)_\ga\times V\to V$ 
given by extending (\ref{modulemult}) linearly.
Conversely, if $V$ is a $(\CC G)_\ga$-module, then
$$ \rho(g)v := g\cdot v, \qquad g\in G,\ v\in V, $$
defines a projective representation $\rho:G\to\GL(V)$
for 
$\ga$.
In the language of modules:
$$\hbox{The $G$-invariant subspaces $V$ are
precisely the $(\CC G)_\ga$-submodules of $V$.} $$
Since $(\CC G)_\ga$ is semisimple, it follows that if
$\rho,\ga$ is a projective representation on $V$, then
$V$ decomposes as direct sum $V=\oplus_j V_j$ of 
irreducible $(\CC G)_\ga$-submodules, i.e., each $\rho|_{V_j}$
is an irreducible projective representation for the cocycle $\ga$. 
When $\rho$ is unitary, the decomposition of $V$ 
can be taken to be an orthogonal direct sum.
There are finitely many such irreducible
representations of $G$ up to equivalence, which we now describe.

The representation $\rho:G\to\GL(\CC^G)$ for $\ga$ given by (\ref{alpharegrep})
generalises the (left) regular (ordinary) representation.
It can also be thought of as a representation on $\CC G$, via the 
indentification of $\CC^G$ with $\CC G$, i.e., $\rho(g) h:=\ga(g,h)gh$.
We will call it the {\bf regular $\ga$-representation for $G$}. As in the 
ordinary case, the regular $\ga$-representation decomposes as a direct
sum of irreducible representations (for the fixed cocycle $\ga$), 
with each irreducible occuring with multiplicity given by its dimension.
In particular, if $R=R_\ga$ is a 
{\bf complete set of irreducible representations}
i.e., each representation occurs exactly once in $R$ up to 
equivalence, then we have
\begin{equation}
\label{nequalsumdsquare}
n =|G|=\dim(\CC^G) =\sum_{\rho\in R} d_\rho^2.
\end{equation}
We 
write $[\rho]$ for the equivalence class of $\rho$ 
(where $\ga$ is fixed)
and $\rho\approx\xi$ for the equivalence,
e.g., $\chi_\rho=\chi_{[\rho]}$ means that $\chi_\rho$ depends only on
$\rho$ up to equivalence. Schur's lemma extends 
(for $(\CC G)_\ga$-homomorphisms between irreducibles), and from it
(see \cite{CH18}), one has

\begin{theorem} Fix $\ga$. If $\rho$ and $\xi$ are irreducible projective 
representations of $G$ 
on $V_1$ and $V_2$, then
$$
 {1\over|G|} \sum_{g\in G} \xi(g)^{-1} L \rho(g)
= {1\over|G|} \sum_{g\in G} {\xi(g^{-1}) L \rho(g)\over\ga(g,g^{-1})\ga(1,1)}
= {\trace(L)\over\dim(V_1)}
\begin{cases}
0, & \rho\not\approx\xi; \cr
I, & \rho=\xi,
\end{cases} $$
for all linear maps $L:V_1\to V_2$.
In particular, if $\rho$ and $\xi$ map to matrices, then
\begin{equation}
\label{rhochiorthog}
\sum_{g\in G} {\xi_{\ell m}(g^{-1})\rho_{jk}(g)\over \ga(g,g^{-1})\ga(1,1)}
= {|G|\over d_\rho}
\begin{cases}
0, &  \rho\not\approx\xi; \cr
\gd_{j\ell}\gd_{km}, 
& \rho=\xi.
\end{cases}
\end{equation}
\end{theorem} 

When $\rho$ is unitary, 
by (\ref{rhoinverse}) we may write 
(\ref{rhochiorthog}) as
\begin{equation}
\label{coordorth}
\inpro{\rho_{jk},\xi_{m\ell}} = {|G|\over d_\rho}
\begin{cases}
0, &  \rho\not\approx\xi; \cr
\gd_{jm}\gd_{k\ell}, 
& \rho=\xi,
\end{cases}
\end{equation}
which we will refer to a the {\it orthogonality of coordinates} 
(of irreducible representations).
This leads to a character theory and Fourier transform for
projective representations.

\section{Character theory for projective representations}
\label{Chartheorysec}

Invariants of the equivalence class 
$$[\rho]=\{T\rho T^{-1}:\hbox{$T$ is invertible}\}$$ 
of a projective representation $\rho:G\to\GL(\CC^{d_\rho})$
include its {\bf $\ga$-character} $\chi_\rho=\chi_{[\rho]}\in\CC^G$
$$ \chi_\rho(g):= \trace(\rho(g))=
\hbox{$\sum_{j=1}^{d_\rho} \rho_{jj}(g)$}, \qquad
g\in G. $$
and the subspace of its coordinates 
$$ U_{\rho,\ga}=U_{[\rho],\ga} 
:= \spam\{\rho_{jk}:1\le j,k\le d_\rho\}. $$
It follows from 
(\ref{nequalsumdsquare}) and (\ref{rhochiorthog}) that
we have the orthogonal decomposition
\begin{equation}
\label{rhofunctorthog}
\CC^G = \bigoplus_{\rho\in R} U_{\rho,\ga}. 
\end{equation}
Since $\chi_\rho\in U_{\rho,\ga}$,  
the $\ga$-characters of inequivalent irreducibles 
are orthogonal. From
$$ \rho(1)=\ga(1,1)I, \qquad \rho(hgh^{-1}) 
= {\ga(hgh^{-1},h)\over\ga(h,g)} \rho(h)\rho(g)\rho(h)^{-1}, $$
it follows that the $\ga$-characters satisfy
\begin{equation}
\label{basicalphacharprops}
\chi_\rho(1)=\ga(1,1)d_\rho,
\qquad \chi_\rho(hgh^{-1}) = {\ga(hgh^{-1},h)\over\ga(h,g)}\chi_\rho(g).
\end{equation}
Furthermore, if $\rho$ is unitary, then (\ref{rhoinverse}) gives
\begin{equation}
\label{alphacharconj}
\chi_\rho(g^{-1}) = \ga(g,g^{-1})\ga(1,1)\overline{\chi_\rho(g)}.
\end{equation}

If $\rho:G\to\GL(V)$ is a representation on $V$, and $V=\oplus_j V_j$
is a decomposition into irreducibles, then
\begin{equation}
\label{chardecomp}
\chi_\rho = \sum_j \chi_{\rho|_{V_j}} = \sum_{\xi\in R} m_\xi \chi_\xi,
\end{equation}
where $m_\xi$ is the multiplicity of the irreducible 
$\xi$. Since the $\ga$-characters of inequivalent irreducibles
are orthogonal, we may determine $m_\xi$ from $\chi_\rho$ via
$$ \inpro{\chi_\rho,\chi_\xi} = m_\xi \inpro{\chi_\xi,\chi_\xi}
= m_\xi \hbox{${|G|\over d_\xi}$} d_\xi = m_\xi |G|. $$
Applying this to the regular $\ga$-represention (\ref{alpharegrep}) 
gives $m_\xi=d_\xi$,
and hence 
(\ref{nequalsumdsquare}).

Motivated by (\ref{basicalphacharprops}),
a function $f:G\to\CC$ is called an {\bf $\ga$-class function} if
$$ f(hgh^{-1}) = {\ga(hgh^{-1},h)\over\ga(h,g)}f(g), \qquad
\forall g,h\in G,$$
and $g\in G$ is a {\bf $\ga$-element} of $G$ if
$$ {\ga(hgh^{-1},h)\over\ga(h,g)}=1,\quad\forall h\in C_G(g) 
\Iff \ga(g,h)=\ga(h,g),\quad \forall h\in {\rm C}_G(g) . $$
It can be shown that the subspace of $\ga$-class functions (which contains
the $\ga$-characters) has an orthogonal basis given by the 
irreducible $\ga$-characters, and its dimension, i.e., the
number of irreducible projective representations (up to equivalence), 
equals the number of conjugacy classes of $G$ which contain a $\ga$-element
(see \cite{CH18} for details).


\section{Fourier analysis} 

Let $R=R_\ga$ be a complete set of inequivalent irreducible projective
representations of a finite group 
$G$ for a given cocycle $\ga$.
We define the {\bf $\ga$-Fourier transform} $F_\ga$ of $f:G\to\CC$
at $\rho\in R$ to be the linear 
map given by
\begin{equation}
\label{projFTdefn}
F_\ga(f)(\rho) 
=(F_\ga f)_\rho:= \sum_{a\in G}{ f(a)\rho(a^{-1})\over\ga(a,a^{-1})\ga(1,1)}.
\end{equation}
When $\rho$ is unitary, i.e., $\rho(a)^*=\rho(a)^{-1}$, this simplifies to 
\begin{equation}
\label{Falphaunitary}
(F_\ga f)_{\rho} = \sum_{a\in G}f(a)\rho(a)^{-1} 
= \sum_{a\in G}f(a)\rho(a)^*.
\end{equation}
We observe that the spectral structure of $(F_\ga f)_{\rho}$ depends only on
$\rho$ up to equivalence, since
\begin{equation}
\label{spectstructequiv}
(F_\ga f)_{T\rho T^{-1}}= T (F_\ga f)_{\rho} T^{-1}.
\end{equation}
To be able to compare with some presentations,
we also define the following variant
\begin{equation}
\label{FtoF1}
(\cF_\ga f)_\rho 
:= \sum_{a\in G}{ f(a)\rho(a)\over\ga(a,a^{-1})\ga(1,1)}
=  (F_\ga \tilde f)_\rho, \quad \tilde f(a):=f(a^{-1}). 
\end{equation}

\begin{example}
By (\ref{rhoinverse}),
the Fourier transforms of the standard basis vectors are
$$ (F_\ga e_g)_\rho = {\rho(g^{-1})\over\ga(g,g^{-1})\ga(1,1)} =\rho(g)^{-1},
\qquad
(\cF_\ga e_g)_\rho = {\rho(g)\over\ga(g,g^{-1})\ga(1,1)} =\rho(g^{-1})^{-1}.
$$
\end{example}

\noindent
We note that $f\mapsto(( F_\ga f)_\rho)_{\rho\in R}$ 
and $f\mapsto((\cF_\ga f)_\rho)_{\rho\in R}$ 
are linear maps
$\CC^G\to\oplus_{\rho} M_{d_\rho}(\CC)$ between
spaces of dimension $n=|G|=\sum_\rho d_\rho^2$.
The corresponding {\bf inverse $\ga$-Fourier transforms} 
at $A=(A_\rho)_{\rho\in R}$ 
are 
given by
\begin{equation}
\label{FtoF2}
(F_\ga^{-1} A)(a):={1\over|G|}\sum_\rho d_\rho \trace(A_\rho\rho(a)),
\qquad (\cF_\ga^{-1} A)(a):=(F_\ga^{-1} A)(a^{-1}). 
\end{equation}

These are inverses of each other, and we can extend the other basic results
of Fourier analysis (for when the multiplier is $\ga=1$), as follows.


\begin{theorem}
The $\ga$-Fourier transform and inverse $\ga$-Fourier transform 
are inverses of each other, and
\begin{equation}
\label{FTconvformula}
F_\ga(\nu) F_\ga(\mu) = F_\ga(\mu *_\ga \nu),
\qquad \cF_\ga(\nu) \cF_\ga(\mu) = \cF_\ga(\nu*_{\tilde\ga} \mu),
\end{equation}
where 
$ \tilde\ga(a,b):=\ga(b^{-1},a^{-1}). $
There is the Plancherel formula
\begin{equation}
\label{Planch2}
\sum_{a\in G} {\nu(a)\mu(a^{-1})\over\ga(a,a^{-1})\ga(1,1)}
= {1\over|G|} \sum_\rho d_\rho 
\trace\bigl((F_\ga \nu)_\rho (F_\ga \mu)_\rho\bigr),
\end{equation}
and when each $\rho\in R$ is unitary
\begin{equation}
\label{Planch1}
\inpro{\nu,\mu} := \sum_{a\in G} \nu(a)\overline{\mu(a)} 
= {1\over|G|} \sum_\rho d_\rho \inpro{(F_\ga\nu)_\rho,(F_\ga\mu)_\rho},
\end{equation}
where $\inpro{A,B}:=\trace(AB^*)$ is the Frobenius inner product on matrices.
\end{theorem}

\begin{proof} Since $F_\ga,F_\ga^{-1}$ and $\cF_\ga,\cF_\ga^{-1}$ 
are linear maps between $n$--dimensional spaces, to prove they 
are inverses, it suffices to prove that
$F_\ga^{-1}F_\ga=I$ and $\cF_\ga^{-1}\cF_\ga=I$.
Moreover, it suffices to prove the just first, since if it holds, 
then (\ref{FtoF1}) and (\ref{FtoF2}) imply
$$ \cF_\ga^{-1} \cF_\ga f (a) = F_\ga^{-1} (\cF_\ga f) (a^{-1})
= F_\ga^{-1} (F_\ga \tilde f) (a^{-1}) = \tilde f (a^{-1})= f(a). $$

We will show that $F_\ga^{-1}F_\ga=I$ on the standard basis $(e_g)_{g\in G}$.
First recall the character of the regular $\ga$-representation is 
$\sum_\rho d_\rho \chi_\rho = \ga(1,1)|G|e_1$. Therefore, we calculate
\begin{align*}
(F_\ga^{-1} & F_\ga e_g)(h)
={1\over|G|}\sum_\rho d_\rho \trace\Bigl(
{ \rho(g^{-1}) \rho(h) \over\ga(g,g^{-1})\ga(1,1)} \Bigr)
={1\over|G|}\sum_\rho d_\rho \trace\Bigl(
{ \ga(g^{-1},h) \rho(g^{-1}h) \over\ga(g,g^{-1})\ga(1,1)} \Bigr) \cr
&={1\over|G|}\sum_\rho d_\rho {\chi_\rho(g^{-1}h)\over\ga(g,g^{-1}h)}
={1\over|G|}M_\ga\bigl(\hbox{$\sum_\rho d_\rho \chi_\rho$}\bigr)_{g,h}
=M_\ga\bigl(\ga(1,1)e_1\bigr)_{g,h}
= I_{g,h} = e_g(h).
\end{align*}
The inversion formula can be expanded 
\begin{equation}
\label{Finvform}
(F_\ga^{-1}F_\ga f)(g)
={1\over|G|}\sum_\rho d_\rho \trace\Bigl(
\sum_{a\in G}{ f(a)\rho(a^{-1})\over\ga(a,a^{-1})\ga(1,1)}
 \rho(g)\Bigr).
\end{equation}

Since both sides of the Plancherel formula (\ref{Planch2}) are
linear in $\nu$ and $\mu$, it suffices (by linearity) to prove 
it for $\nu=e_g$ and $\mu=e_h$.
Using (\ref{Finvform}), we have
\begin{align*}
{1\over|G|}\sum_\rho d_\rho \trace\bigl((F_\ga e_g)_\rho(F_\ga e_h)_\rho\bigr)
&= {1\over|G|}\sum_\rho d_\rho \trace\Bigl(
\sum_{a\in G}{ e_g(a)\rho(a^{-1})\over\ga(a,a^{-1})\ga(1,1)}
 {\rho(h^{-1})\over\ga(h,h^{-1})\ga(1,1)}\Bigr) \cr
&={e_g(h^{-1})\over\ga(h,h^{-1})\ga(1,1)}
=\sum_{a\in G}{e_g(a)e_h(a^{-1})\over\ga(a,a^{-1})\ga(1,1)}.
\end{align*}
Since both sides of (\ref{Planch1}) are
linear in $\nu$ and conjugate linear in $\mu$, it
again suffices to consider $\nu=e_g$ and $\mu=e_h$.
Since $\rho$ is unitary,  $(F_\ga e_h)_\rho^*=(\rho(h)^{-1})^*=\rho(h)$,
and so 
\begin{align*}
{1\over|G|}\sum_\rho d_\rho \inpro{(F_\ga e_g)_\rho,(F_\ga e_h)_\rho\bigr)}
&= {1\over|G|}\sum_\rho d_\rho \trace\Bigl(
\sum_{a\in G}{ e_g(a)\rho(a^{-1})\over\ga(a,a^{-1})\ga(1,1)}
 \rho(h)\Bigr) \cr
&=e_g(h)
=\sum_{a\in G}e_g(a)\overline{e_h(a)}.
\end{align*}

We now prove the convolution formulas.
On one hand, we have
$$ \bigl( F_\ga (\nu*_\ga \mu)\bigr)_\rho = \sum_{a\in G} \sum_{t\in G} 
{ \nu(a t^{-1})\mu(t) \over \ga(at^{-1},t)} \rho(a)^*,  $$
Since
$\rho(a)^*\rho(b)^* = (\rho(b)\rho(a))^*
= (\ga(b,a)\rho(ba))^*= {\rho(ba)^*\over\ga(b,a)}$, 
and we obtain
\begin{align*}
(F_\ga \nu)_\rho (F_\ga\mu)_\rho
&= \sum_{a\in G} \nu(a)\rho(a)^* \sum_{b\in G} \mu(b)\rho(b)^*
= \sum_{a\in G} \sum_{b\in G} 
{ \nu(a) \mu(b) \over\ga(b,a)} \rho(ba)^*\cr
&= \sum_{c\in G} \sum_{t\in G} {\nu(t) \mu(c t^{-1}) \over\ga(ct^{-1},t)} \rho(c)^*
= \bigl(F_\ga(\mu*_\ga\nu)\bigr)_\rho,
\end{align*}
which gives the first convolution formula. From this, we have
$$ \cF_\ga(\tilde\nu) \cF_\ga(\tilde\mu) 
= \cF_\ga(\widetilde{\mu *_\ga \nu}), $$
and a calculation gives
\begin{align*}
\widetilde{\mu *_\ga \nu}(g)
&= (\mu*_\ga \nu)(g^{-1}) := \sum_{t\in G} 
{ \mu(g^{-1} t^{-1})\nu(t) \over \ga(g^{-1}t^{-1},t)}
= \sum_{t\in G} { \tilde\mu(tg)\tilde\nu(t^{-1}) \over 
\tilde\ga(t,g^{-1}t^{-1})} 
= \sum_{s\in G} 
{ \tilde\nu(gs^{-1}) \tilde\mu(s) \over \tilde\ga(sg^{-1},s^{-1})} 
\cr
&
= (\tilde\nu *_{\tilde\ga} \tilde\mu) (g),
\end{align*}
which gives the second convolution formula.
\end{proof}

To the best of our knowledge, (\ref{projFTdefn}) is the first time
that a Fourier transform has been defined for projective
representations. For ordinary representations,
the Fourier transform is well studied, see, e.g., \cite{D88}, \cite{T99}.

\begin{example}
\label{FTsquare}
The condition for $P=M_\ga(\nu)$ to satisfy $P^2=P$ is 
$ \nu*_\ga \nu =\nu, $
which transforms to 
$$ F_\ga (\nu*_\ga \nu) = F_\ga(\nu)^2=F_\ga(\nu). $$
i.e., each $(F_\ga \nu)_\rho$ satisfies this condition.
\end{example}

\begin{lemma}
\label{rhoFTformula}
If $\rho:G\to\GL(\CC^{d_\rho})$ is a {\it unitary} irreducible
projective representation, then
$$ (F_\ga \rho_{rs})_\xi 
= {|G|\over d_\rho}
\begin{cases}
0, & \xi\not\approx\rho; \cr
e_s^* e_r, & \xi=\rho.
\end{cases} $$ 
\end{lemma}

\begin{proof}
In view of (\ref{spectstructequiv}), it suffices to prove this
when all the $\xi\in R$ are unitary. Here, the orthogonality of
coordinates (\ref{coordorth}) gives
$$ (F_\ga \rho_{rs})_\xi 
= \sum_{a\in G} \rho_{rs}(a)\xi(a)^*
= \sum_{a\in G} \rho_{rs}(a)(\overline{\xi_{kj}(a)})_{j,k=1}^{d_\xi}
= (\inpro{\rho_{rs},\xi_{kj}})_{j,k=1}^{d_\xi} 
 = \begin{cases}
0, & \xi\not\approx\rho; \cr
{|G|\over d_\rho} e_s^*e_r, & \xi=\rho.
\end{cases} $$ 
\end{proof}

\begin{example}
As 
examples, 
we have (for $\rho$ unitary or not) that
the $\ga$-character satisfies
$$ (F_\ga \chi_\rho)_\xi 
 = {|G|\over d_\rho}
 \begin{cases}
0, & \xi\not\approx\rho; \cr
I, & \xi=\rho,
\end{cases} $$ 
and if
$$ f(g)=\trace(\rho(g)A)=\inpro{\rho(g),A^*}
=\sum_{j,k} \rho_{jk}(g) a_{kj}, $$
then
\begin{equation}
\label{FTirredA}
 (F_\ga f)_\xi 
 = {|G|\over d_\rho}
 \begin{cases}
0, & \xi\not\approx\rho; \cr
A, & \xi=\rho.
\end{cases}
\end{equation}
In particular, for $f(g)=\inpro{\rho(g)v,v}=\trace(\rho(g)vv^*)$, i.e., $A=vv^*$ above, 
we have
\begin{equation}
\label{FTirredGram}
 (F_\ga f)_\xi 
 = {|G|\over d_\rho}
 \begin{cases}
0, & \xi\not\approx\rho; \cr
vv^*, & \xi=\rho.
\end{cases}
\end{equation}
\end{example}

\begin{lemma}
\label{FThermitiancdn}
For $\ga$ unitary, a $(G,\ga)$-matrix $M_\ga(\nu)$ is Hermitian
if and only if 
$$ F_\ga(\nu)_\rho^* = F_\ga(\nu)_\rho,\qquad \rho\in R. $$
\end{lemma}

\begin{proof} We have $M_\ga(\nu)^*=M_\ga(\nu)$, i.e.,
$\nu^{*,\ga}=\nu$, 
 if and only if
 $F_\ga(\nu^{*,\ga})_\rho= F_\ga(\nu)_\rho$, 
for all $\rho$.
Using (\ref{rhoinverse}), we have
\begin{align*}
F_\ga(\nu^{*,\ga})_\rho
&= \sum_{a\in G} \overline{\nu(a^{-1})}\ga(a,a^{-1})\ga(1,1)\rho(a)^*
= \Bigl(\sum_{a\in G} \nu(a^{-1})
{\rho(a) \over\ga(a,a^{-1})\ga(1,1)} \Bigl)^* \cr
&= \Bigl(\sum_{a\in G} \nu(a^{-1})
\rho(a^{-1})^{*} \Bigl)^*
= F_\ga(\nu)_\rho^*,
\end{align*}
which gives the result.
\end{proof}

Example \ref{FTsquare} and Lemmas \ref{rhoFTformula} and \ref{FThermitiancdn}
are sufficient to obtain the characterisation of tight $(G,\ga)$-frames 
given in \S \ref{Gframecharchap}.
Before doing this, we give more precise results about the spectral 
structure of $(G,\ga)$-matrices, which are both enlightening
and useful for other applications. 

\begin{definition}
If $(F_\ga\nu)_\xi=0$, for each $\xi\not\approx\rho$,
then we say that $\nu\in\CC^G$ is a {\bf $\rho$-function} and 
$M_\ga(\nu)$ is a {\bf $\rho$-matrix}.
\end{definition}

The vector spaces of $\rho$-functions and $\rho$-matrices depend only 
$\rho$ up to equivalence, and 
have dimension $d_\rho^2$.
It follows from the convolution formula (\ref{FTconvformula}) for $F_\ga$
that 
\begin{itemize}
\item The $\rho$-functions are closed under the $*_\ga$ convolution.
\item The product of $\rho$-matrices is a $\rho$-matrix.
\end{itemize}

\begin{lemma} The following are equivalent
\begin{enumerate}
\item $M_\ga(\nu)$ is a $\rho$-matrix.
\item $\nu$ is a $\rho$-function.
\item $\nu\in\spam\{\rho_{jk}:1\le j,k\le d_\rho\}$.
\item $\nu\perp \xi_{jk}$, $1\le j,k\le d_\xi$, $\forall\xi\in R$, $\xi\ne\rho$.
\end{enumerate}
\end{lemma}

\begin{proof} Let $B=(F_\ga \nu)_\rho$.
Then $\nu$ is a $\rho$-function if
and only if $(F_\ga \nu)_\xi=0$, $\xi\not\approx\rho$, i.e.,
$$\nu(g) = {1\over |G|} \sum_\xi d_\xi \trace( (F_\ga\nu)_\xi\xi(g))
= {d_\rho\over|G|} \trace(B\rho(g))
= {d_\rho\over|G|} \sum_{j,k} {b_{kj}} \rho_{jk}(g). $$
This and the fact and the orthogonality of coordinates (\ref{coordorth})
gives the result.
\end{proof}

From $F_\ga^{-1}F_\ga=I$ and (\ref{rhofunctorthog}), 
we have that 
each $f\in\CC^G$ can be uniquely decomposed 
as a sum of orthogonal $\rho$-functions
\begin{equation}
\label{rhofunctdecomp}
f=\sum_\rho f_\rho, \qquad
f_\rho =f_{[\rho]}:= {d_\rho\over|G|} \trace( (F_\ga f)_\rho\rho),
\end{equation}
which we will call
{\bf $F_\ga$-Fourier decomposition}
of $f\in\CC^G$ into orthogonal $\rho$-functions.
We will also call
\begin{equation}
\label{Mfunctdecomp}
M_\ga(f) = \sum_\rho M_\ga(f_\rho),
\end{equation}
the {\bf $F_\ga$-Fourier decomposition} of 
$M_\ga(f)$ into $\rho$-matrices.
The following lemma shows that
\begin{equation}
\label{rhomatprods}
M_\ga(f_\rho) M_\ga(f_\xi) = 0, \quad \rho\not\approx\xi.
\end{equation}

\begin{lemma}
\label{prodofrhomats}
For $\rho,\xi\in R$, we have
$$  \rho_{jk}*_\ga\xi_{rs}
=\sum_\ell \inpro{\xi_{rs},\rho_{k\ell}}\, \rho_{j \ell}
= {|G|\over d_\rho} \begin{cases}
 \rho_{js}, & \xi=\rho,\  r=k; \cr
0, & \hbox{otherwise}.
\end{cases}
$$
\end{lemma}

\begin{proof}
We have
$$ (\rho_{jk}*_\ga\xi_{rs})(g)
=\sum_{t\in G} {\rho_{jk}(gt^{-1})\xi_{rs}(t)\over\ga(gt^{-1},t)}. $$
Since
$$ \rho(gt^{-1}) 
= {\rho(g)\rho(t^{-1})\over\ga(g,t^{-1})}
= {\rho(g)\ga(t,t^{-1})\ga(1,1)\rho(t)^*\over\ga(g,t^{-1})}, $$
and 
$\ga(g,t^{-1}) \ga(gt^{-1},t) = \ga(g,t^{-1}t)\ga(t^{-1},t)=
\ga(1,1) \ga(t,t^{-1})$,
we have
\begin{align*}
(\rho_{jk}*_\ga\xi_{rs})(g)
&=\sum_{t\in G} { \ga(t,t^{-1})\ga(1,1) \over\ga(g,t^{-1})\ga(gt^{-1},t)}
(\rho(g)\rho(t)^*)_{jk} \xi_{rs}(t) \cr
& =\sum_{t\in G} 
\sum_\ell \rho_{j\ell}(g)\overline{\rho_{k\ell}(t)} \xi_{rs}(t)
= \sum_\ell \inpro{\xi_{rs},\rho_{k\ell}} \, \rho_{j\ell}(g),
\end{align*}
and the orthogonality completes the result.
\end{proof}

Further, if $\ga$ is unitary, then $f_\rho^{*,\ga}$ is also a $\rho$-function,
since Example \ref{starrhojk} gives
$$ f_\rho=\sum_{j,k} a_{jk}\rho_{jk} \Implies
f_\rho^{*,\ga}
= \sum_{j,k} \overline{a_{jk}} \rho_{jk}^{*,\ga}
= \sum_{j,k} \overline{a_{jk}} \rho_{kj}, $$
so the $\rho$-matrices are closed under the Hermitian transpose,
and from (\ref{rhomatprods}) we obtain
\begin{equation}
\label{rhomatinpro}
\inpro{M_\ga(f_\rho), M_\ga(f_\xi)}
= \trace(M_\ga(f_\rho)M_\ga(f_\xi^{*,\ga}))
=0, \quad \rho\not\approx\xi.
\end{equation}

\begin{example}
From 
(\ref{Mfunctdecomp})
(\ref{rhomatprods})
and (\ref{rhomatinpro}),
we have
$$ M_\ga(\nu)M_\ga(\mu)
= \sum_\rho M_\ga(\nu_\rho)M_\ga(\mu_\rho), \qquad
M_\ga(f)^k = \sum_\rho M_\ga(f_\rho)^k, $$
$$  \inpro{M_\ga(\nu),M_\ga(\mu)}
= |G|\inpro{\nu,\mu}
= |G|\sum_\rho \inpro{\nu_\rho,\mu_\rho}
= \sum_\rho\inpro{M_\ga(\nu_\rho),M_\ga(\mu_\rho)}. $$
\end{example}



\section{The spectral structure of the $(G,\ga)$--matrices}

The circulant matrices, i.e., the $(G,\ga)$-matrices for $G$ a
cyclic group and $\ga=1$, are all simultaneously unitarily diagonalisable
by the Fourier matrix, i.e., the characters (representations) of $G$
are the eigenvectors of the circulant matrices.
We now investigate to what extent this result extends to 
general $(G,\ga)$--matrices.

We recall from \S\ref{Chartheorysec} the orthogonal 
decomposition of $\CC^G$ into $\rho$-functions
$$ \CC^G = \bigoplus_{\rho\in R} U_{\rho,\ga},
\qquad
U_{\rho,\ga} := \spam\{\rho_{jk}:1\le j,k\le d_\rho\} . $$
For ordinary representations, i.e., $\ga=1$, we now present
the standard diagonalisation result (see \cite{D88}, \cite{J18}), which shows 
that the $U_{\rho,\ga}$ are invariant subspaces of the $(G,\ga)$-matrices.

Suppose henceforth that each $\rho\in R$ is unitary, 
and let $E$ be the unitary matrix
$$ E=E_R:=[ E_{\rho,k}:\rho\in R, 1\le k \le d_\rho], 
\qquad 
E_{\rho,k}:=
\hbox{$\sqrt{d_\rho\over|G|}$} 
[\rho_{k1},\rho_{k2},\ldots,\rho_{kd_\xi}]. $$

\begin{theorem} 
\label{ordinaryblockdiag}
For ordinary representations, 
the matrix
$E^* M_\ga(\nu)E$ is block diagonal, with diagonal blocks
$(A_\rho:\rho\in R,1\le k \le d_\rho)$, where
$A_\rho:=(\cF_\ga\nu)_\rho$.
\end{theorem}

\begin{proof} 
By way of motivation, 
if we write $E=[\xi_1,\xi_2,\ldots]$ 
and $M_\ga(v) E= E\gL$, then
$$ M_\ga(v)_{g,h} = (E\gL E^*)_{g,h}
= \sum_s\sum_t E_{g,s} \gL_{st} (E^*)_{t,h}
= \sum_s\sum_t \xi_s(g) \gL_{st} \overline{\xi_t(h)}.$$
By Fourier inversion, we calculate
$$ M_\ga(\nu)_{g,h} = \nu(g^{-1}h)
= (\cF_\ga^{-1}\cF_\ga \nu)(g^{-1}h)
= (F_\ga^{-1}\cF_\ga \nu)(h^{-1}g)
={1\over|G|} \sum_\rho d_\rho \trace\bigl(A_\rho \rho(h^{-1}g)\bigr).$$
Since the representations are ordinary, i.e., $\ga=1$,
$\trace\bigl(A_\rho \rho(h^{-1}g)\bigr)
= \trace\bigl(\rho(g)A_\rho \rho(h)^*\bigr).$
Hence, by writing $A_\rho=[a_{jk}^\rho]_{j,k=1}^{d_\rho}$, we obtain
$$ M_\ga(\nu)_{g,h} 
={1\over|G|} \sum_\rho d_\rho 
\sum_j (\rho(g) A_\rho \rho(h)^*)_{jj}
= \sum_\rho 
 \sum_j \sum_s\sum_t
{ d_\rho \over|G|} 
 \rho_{js}(g) a^\rho_{st} \overline{\rho_{jt}}(h), $$
which gives the result.
\end{proof}

From the above, it follows that the orthogonal subspaces
$$ U_{\rho,\ga,j} := \spam\{\rho_{jk}:1\le k\le d_\rho\}, \qquad
\rho\in R, \quad 1\le j \le d_\rho, $$
are invariant subspaces of the $(G,\ga)$-matrices when $\ga=1$.
These subspaces do not give a unique orthogonal decomposition of 
$U_{\rho,\ga}$ into invariant subspaces of the $(G,\ga)$-matrices,
since, for any unitary $T$, one has
\begin{equation}
\label{Urhoalphajnotunique}
U_{\rho,\ga} = \bigoplus_{j=1}^{d_\rho} U_{T\rho T^{-1},\ga,j}.
\end{equation}

It is natural to suppose that the $U_{\rho,\ga,j}$ are invariant
subspaces of the $(G,\ga)$-matrices for the projective case also, 
and to adapt the argument above to prove it. Here
$$ M_\ga(\nu)_{g,h} 
= \sum_\rho 
 \sum_j \sum_s\sum_t {\ga(h,h^{-1}g)\over\ga(g,g^{-1}h)}
{ d_\rho \over|G|} 
a^\rho_{st} \rho_{js}(g) \overline{\rho_{jt}}(h), $$
and so the remainder of the argument breaks down.

To understand the invariant subspaces of the $(G,\ga)$-matrices,
we first consider the range of the $(G,\ga)$-matrices $M_\ga(\rho_{jk})$.

\begin{lemma} Let $\rho,\xi\in R$. Then
$$ M_\ga(\rho_{jk}) v 
= \sum_{\ell =1}^{d_\rho} \sum_{h\in G}  
\rho_{\ell k}(h) v_h\, \overline{\rho_{\ell j}},\qquad v\in\CC^G. $$
and, in particular, 
$$ M_\ga(\rho_{jk}) \overline{\xi_{st}}
= {|G|\over d_\rho} 
\begin{cases}
\overline{\rho_{sj}} , & \xi=\rho, \ t=k; \cr
0, & \hbox{otherwise}.
\end{cases}
$$
\end{lemma}
 
\begin{proof}
Since $\rho(g^{-1}h) =\ga(g,g^{-1}h)\rho(g)^*\rho(h)$, 
we have
\begin{align*}
 (M_\ga(\rho_{jk})v)_g  
&= \sum_h M_\ga(\rho_{jk})_{g,h} v_h
= \sum_h {\rho_{jk}(g^{-1}h)\over\ga(g,g^{-1}h)} v_h
= \sum_h (\rho(g)^*\rho(h))_{jk} v_h \cr
&= \sum_h \sum_\ell (\rho(g)^*)_{j\ell}\rho(h)_{\ell k} v_h
= \sum_h \sum_\ell \overline{\rho_{\ell j}(g)}\rho_{\ell k}(h) v_h.
\end{align*}
Hence
$$ M_\ga(\rho_{jk}) \overline{\xi_{st}}
= \sum_\ell \inpro{\rho_{\ell k},\xi_{st}} \,\overline{\rho_{\ell j}}
=\begin{cases}
{|G|\over d_\rho} \overline{\rho_{sj}} , & \xi=\rho, \ t=k; \cr
0, & \hbox{otherwise},
\end{cases}
$$
since the entries of $\rho$ and $\xi$ are orthogonal.
\end{proof}

We therefore conclude that the orthogonal subspaces
$$ V_{\rho,\ga,j} := \spam\{\overline{\rho_{jk}}: 1\le k\le d_\rho\}
=\overline{U_{\rho,\ga,j}}
,
\qquad \rho\in R, \quad 1\le j\le d_\rho,$$
are invariant subspaces of the $(G,\ga)$-matrices.

This gives the desired simultaneous unitary (block) diagonalisation
of projective group matrices:

\begin{theorem} 
\label{Galphamatdiag}
For projective representations, 
the matrix
$\overline{E}^* M_\ga(\nu)\overline{E}$ is block diagonal, 
with diagonal blocks
$(B_\rho^T:\rho\in R,1\le k \le d_\rho)$, where
$B_\rho:=(F_\ga\nu)_\rho$.
\end{theorem}

\begin{proof}
Let $(F_\ga\nu)_\rho = B_\rho = [b_{jk}^\rho]_{j,k=1}^{d_\rho}$. Then
\begin{align*}
M_\ga(\nu)_{g,h} 
&= {\nu(g^{-1}h)\over\ga(g,g^{-1}h)}
 = {(F_\ga^{-1} F_\ga\nu)(g^{-1}h)\over\ga(g,g^{-1}h)}
 = {1\over\ga(g,g^{-1}h)}
{1\over|G|} \sum_\rho d_\rho \trace\bigl( B_\rho \rho(g^{-1}h)\bigr).
\end{align*}
Since
$ \rho(g^{-1}h) 
= \ga(g,g^{-1}h) \rho(g)^*\rho(h) , $
we obtain
$$ M_\ga(\nu)_{g,h} 
= \sum_\rho {d_\rho\over|G|} \trace\bigl( \rho(h) B_\rho \rho(g)^*\bigr)
= \sum_\rho {d_\rho\over|G|} 
\sum_j\sum_s\sum_t \rho_{js}(h) b_{st}^\rho \overline{\rho_{jt}(g)}. $$
Thus, we calculate 
\begin{align*}
(\overline{E_{\xi,k_\xi}}^* M_\ga(\nu) \overline{E_{\eta,k_\eta}} )_{\ell m}
&= \sum_g\sum_h {\sqrt{d_\xi}\over\sqrt{|G|}} {\xi_{k_\xi \ell}(g)}  
\sum_\rho {d_\rho\over|G|} \sum_{j,s,t}
\rho_{js}(h) b_{st}^\rho \overline{\rho_{jt}(g)}
{\sqrt{d_\eta}\over\sqrt{|G|}} \overline{\eta_{k_\eta m}}(h) \cr
&= {\sqrt{d_\xi d_\eta}\over|G|}
\sum_\rho {d_\rho\over|G|} \sum_{j,s,t}
\inpro{\xi_{k_\xi \ell},\rho_{jt}}
\inpro{\rho_{js},\eta_{k_\eta m}}
b_{st}^\rho 
=\begin{cases}
b_{ml}^\xi, & \xi=\eta; \cr
0, &\hbox{otherwise},
\end{cases}
\end{align*}
which gives the result.
\end{proof}

\begin{example}
For ordinary representations,
$\overline{\rho}(g)\overline{\rho}(h)
=\overline{\rho(g)\rho(h)} =\overline{\rho(gh)}
=\overline{\rho}(gh)$,
so that $\{\overline{\rho}\}_{\rho\in R}$ is
another complete set of ordinary representations for $G$,
with $E_{\overline{\rho},k}=\overline{E_{\rho,k}}$. Hence
from 
Theorem \ref{ordinaryblockdiag},
we have
$$ \overline{E}^* M_\ga(\nu) \overline{E}
=\diag(A_{\overline{\rho}}:\rho\in R,1\le k\le d_\rho), $$
where $A_{\overline{\rho}} = \cF_\ga(\nu)_{\overline{\rho}}
= F_\ga(\nu)_\rho^T = B_\rho^T$
(which is Theorem \ref{Galphamatdiag}).
%
\end{example}

\begin{example} The determinant of a $(G,\ga)$-matrix factors
$$ \det(M_\ga(\nu))=\det(\overline{E}^* M_\ga(\nu) \overline{E})
=\prod_\rho \det(B_\rho)^{d_\rho}
=\prod_{\rho\in R} \det( (F_\ga \nu)_\rho)^{d_\rho}. $$
This ``factorisation of the group determinant'' was one
of the motivations which lead to the development of representation 
theory (see \cite{J18}).
\end{example}

\begin{example}
The blocks of the diagonal form of a $(G,\ga)$-matrix
$M_\ga(f)$ are unique up to similarity, and in particular, the 
Jordan canonical form is given by the block diagonal matrix
with $d_\rho$ diagonal blocks given by the Jordan canonical
form of $(F_\ga f)_\rho$.
\end{example}

For ordinary representations, invariant subspaces
of the $(G,\ga)$-matrices are
$$ \{U_{\rho,\ga,j}\}_{\rho\in R, 1\le j\le d_\rho} 
= \{V_{\rho,\ga,j}\}_{\rho\in R, 1\le j\le d_\rho}. $$
For projective representations, the $\{V_{\rho,\ga,j}\}$
are invariant subspaces (with a neat block diagonal form).
It is no longer the case that even
$\{V_{\rho,\ga}\}_{\rho\in R} = \{U_{\rho,\ga}\}_{\rho\in R}$,
see (\ref{UrhoalphaVrhoalpharelat}).
Nevertheless, in the specific cases considered so far
(see \S\ref{dihedralgroupsect}), it 
appears that there is a block diagonal form for the 
$\{\cU_{\rho,\ga}\}_{\rho\in R}$
which is more complicated (the blocks are no longer ordered by $\rho$,
and blocks similar to the same $(F_\ga\nu)_\rho$ are not all identically
equal).


From the orthogonal decomposition
\begin{equation}
\label{rhofunctorthogfine}
\CC^G = \bigoplus_{\rho\in R}\bigoplus_{j=1}^{d_\rho} U_{\rho,\ga,j}, 
\end{equation}
we obtain a {\bf fine scale} $F_\ga$-Fourier decomposition
$$ M_\ga(f) = \sum_\rho\sum_{j=1}^{d_\rho} M_\ga (f_{\rho,j}), \qquad
f_{\rho,j}= {d_\rho\over|G|} \trace( (F_\ga f)_\rho e_je_j^* \rho) 
\in U_{\rho,\ga,j}. $$
We observe from 
(\ref{Urhoalphajnotunique}), or the formula
$ f_{\rho,j}(g) = \inpro{\rho(g)(F_\ga f)_\rho e_j,e_j}, $
that there is not a unique {\it fine scale} $F_\ga$-Fourier decomposition.

Let $M_{\rho,\ga,j}$ be the $d_\rho$-dimensional 
vector space of $(G,\ga)$-matrices of the form $M_\ga(f_{\rho,j})$. 
It follows from Lemma \ref{prodofrhomats}
(or the block diagonal form) that $M_{\rho,\ga,j}$ is closed
under multiplication, indeed $M_{\rho,\ga,j}M_{(G,\ga)} \subset M_{\rho,\ga,j}$,
and hence is an algebra. 
Moreover, from Lemma \ref{prodofrhomats}, we have
$M_\ga(f_{\rho,j})M_\ga(f_{\xi,k}) = 0$, $(\rho,j)\ne(\xi,k)$, 
and consequently
$$  M_\ga(\nu)M_\ga(\mu) = \sum_{\rho\in R}\sum_{j=1}^{d_\rho}
 M_\ga(\nu_{\rho,j})M_\ga(\mu_{\rho,j}). $$

We now show the Fourier decompositions of $M_\ga(f)$ are
into {\it low rank} $(G,\ga)$-matrices.

\begin{proposition} 
\label{Malpharankprop}
The rank of a $(G,\ga)$-matrix satisfies
\begin{equation}
\label{MalpharankI}
\rank(M_\ga(f)) 
= \sum_\rho \rank(M_\ga(f_\rho))
= \sum_\rho d_\rho \rank((F_\ga f)_\rho),
\end{equation}
\begin{equation}
\label{MalpharankII}
\rank(M_\ga(f)) 
\le \sum_\rho\sum_{j} \rank(M_\ga(f_{\rho,j})),
\qquad \rank(M_\ga(f_{\rho,j}))\in\{0,d_\rho\}.
\end{equation}
In particular, $M_\ga(f)$ is invertible if and only if $f_{\rho,j}\ne0$,
$\forall\rho,j$.
\end{proposition}

\begin{proof}
The block diagonal matrix $\overline{E}^* M_\ga(f) \overline{E}$ has
diagonal blocks $(F_\ga f)_\rho^T=(F_\ga f_\rho)_\rho^T$, each repeated
$d_\rho$ times, and so we have
$$ \rank(M_\ga(f)) 
= \sum_\rho d_\rho \rank((F_\ga f)_\rho)
= \sum_\rho d_\rho \rank((F_\ga f_\rho)_\rho). $$
The block diagonal form of $M_\ga(f_\rho)$ has just $d_\rho$ (possibly) 
nonzero blocks $F_\ga(f_\rho)_\rho^T$, and so the rank of $M_\ga(f_\rho)$
is $d_\rho\rank((F_\ga f_\rho)_\rho)$.
The block diagonal form of $M_\ga(f_{\rho,j})$ has $d_\rho$ (possibly) 
nonzero blocks 
$$ (F_\ga f_{\rho,j})_\rho^T 
= ( (F_\ga f)_\rho e_je_j^*)^T
= e_je_j^* (F_\ga f_\rho)_\rho^T, $$
which either have rank one (since $\rank(e_je_j^*)=1$),
or are rank zero (when $f_{\rho,j}=0$).
\end{proof}

\begin{example}
Proposition \ref{Malpharankprop}
gives a restriction on the possible rank of a $(G,\ga)$-matrix,
e.g., for $G=D_{4m}$ (the dihedral group of order $4m$) and 
$\ga$ the nontrivial cocycle, all irreducibles have  $d_\rho=2$
(see \S\ref{dihedralgroupsect}), 
and so a $(G,\ga)$-matrix must be of even rank.
\end{example}

There is some interest in subspaces (in our case subalgebras)
of matrices with a restriction on their rank \cite{DGMS10}.


\section{The characterisation of tight $(G,\ga)$-frames.}
\label{Gframecharchap}

We now give the main application of our results: a simple characterisation 
and explicit construction of all tight $(G,\ga)$-frames.
Let $R$ be a complete set of irreducible unitary representations for a (unitary) cocycle $\ga$. With very little effort we have:

\begin{lemma}
\label{Galphatightlemma}
Let  $P=M_\ga(\nu)$ be a $(G,\ga)$-matrix. Then the following are equivalent
\begin{enumerate}
\item 
 $P$ is the Gramian of a normalised tight $(G,\ga)$-frame.
\item $P$ is an orthogonal projection.
\item Each $F_\ga(\nu)_\rho$, $\rho\in R$, is an orthogonal projection.
\end{enumerate}
\end{lemma}

\begin{proof}
Since a sequence of vectors is a normalised tight frame if and only if
its Gramian is an orthogonal projection, we have the equivalence of
the first two. The condition that $P$ be an orthogonal projection, 
i.e., $P^2=P$ and $P^*=P$, is equivalent to the last condition by
Example \ref{FTsquare} and Lemma \ref{FThermitiancdn}.
\end{proof}

We will refer to $(F_\ga(\nu)_\rho)_{\rho\in R}$ as the {\bf Fourier
coefficients} of $P=M_\ga(\nu)$,
 or of any $(G,\ga)$-frame with Gramian $M_\ga(\nu)$.

\begin{example} If $G$ is abelian and $\ga=1$, then all the irreducibles
have dimension $1$, and so the Fourier coefficients of a normalised 
tight $G$-frame
must be $0$ or $1$. Thus there are a finite number of such $G$-frames
for $G$ abelian, the so called 
{\bf harmonic frames}.
\end{example}

Now we suppose that $\rho:G\to V$ is a unitary action on $V$ (for
a cocycle $\ga$). We now answer the question:
 when is $(\rho(g)v)_{g\in G}$ a normalised tight $(G,\ga)$-frame for $V$? 
The result below is the main structure theorem for tight $(G,\ga)$-frames,
and was first given in \cite{VW05} (Theorem 6.18) for the ordinary case,
and in \cite{CH18} (Theorem 2.11) for the projective case. The proof
given here uses the Fourier coefficients of the frame,
rather than asserting
(\ref{normalisedtfdefn}) for each irreducible,
and gives insight into the result.

\begin{theorem} Let $\rho:G\to V$ be a unitary action on $V$ for
a cocycle $\ga$, and $V=\oplus V_j$ be an orthogonal direct
sum of irreducible $(\CC G)_\ga$-modules. If $v=\sum_j v_j$, $v_j\in V_j$,
then $(gv)_{g\in G}=(\rho(g)v)_{g\in G}$ 
is a normalised tight $(G,\ga)$-frame for $V$
if and only if
$$ \norm{v_j}^2 = {\dim(V_j)\over|G|}, \quad\forall j, $$
and in the case $V_j$ is $(\CC G)_\ga$-isomorphic to $V_k$, $j\ne k$,
via $\gs:V_j\to V_k$ that
$$ \inpro{\gs v_j,v_k}=0. $$
\end{theorem}

\begin{proof} By Lemma \ref{Galphatightlemma},
$\Phi:=(gv)_{g\in G}$ is a normalised tight
frame for $V$ 
 if and only if its Fourier coefficients 
are orthogonal projections, 
and the sum of the ranks of these projections
is $\dim(V)$. We now calculate the Fourier coefficients $(c_\xi)_{\xi\in R}$.

Let $\rho_j:V_j\to V_j$ be the irreducible representation 
given by $\rho_j(g):=\rho(g)|_{V_j}$,
and let $\gs_j:V_j\to V_\xi$ be 
a unitary $(\CC G)_\ga$-isomorphism to the $\xi:G\to\GL(V_\xi)$ in $R$ 
with 
$\rho_j\approx\xi$. 
We note that $\xi=\gs_j\rho_j\gs_j^{-1}$.
By the orthogonality of the $V_j$, the Gramian of $(gv)_{g\in G}$
is the sum of the Gramians $M_\ga(f_j)$ of the $(G,\ga)$-frames
$\Phi_j:=(\rho_j(g)v_j)_{g\in G}$ for $V_j$, where
\begin{align*}
f_j(g) 
&:= \inpro{\rho_j(g)v_j,v_j}
= \trace(\rho_j(g)v_jv_j^*)
= \trace(\gs_j\rho_j(g)\gs_j^{-1} \gs_j v_j (\gs_j v_j)^*) \cr
&= \trace(\xi(g) \gs_j v_j (\gs_j v_j)^*). 
\end{align*}
It follows from (\ref{FTirredGram}) that the Fourier coefficients of $\Phi_j$
(given by the $\xi$-function $f_j$) are
$$ (F_\ga f_j)_\eta 
 = 
 \begin{cases}
0, & \eta\not\approx\xi; \cr
w_jw_j^*, & \eta=\xi,
\end{cases}
\qquad w_j := \hbox{$\sqrt{|G|\over d_\xi}$} \gs_jv_j\quad
(\rho_j\approx\xi), $$
and hence the Fourier coefficients of $\Phi$ are
$$ c_\xi := \sum_{j:\rho_j\approx\xi} w_jw_j^*. $$

If $\Phi$ is a normalised tight frame for $V$, then each $\Phi_j$ must
be one for 
$V_j$ (since orthogonal projections map
normalised tight frames to normalised tight frames), i.e., $w_jw_j^*$
is a rank one orthogonal projection, which gives
$$ \norm{w_j}^2= {|G|\over d_\xi} \norm{\gs_j v_j}^2 = {|G|\over d_\xi} \norm{v_j}^2=1
\Iff \norm{v_j}^2={\dim(V_j)\over|G|}. $$
Finally, for $c_\xi$ to be an orthogonal projection, we need 
$w_j\perp w_k$, $j\ne k$, $\rho_j,\rho_k\approx\xi$, i.e.,
$$ w_j\perp w_k 
\Iff \gs_j v_j \perp \gs_k v_k
\Iff \gs_k^{-1}\gs_j v_j \perp v_k
\Iff \gs v_j \perp v_k, $$
where, by Schur's lemma, we can replace the $(\CC G)_\ga$-isomorphism 
$\gs_k^{-1}\gs_j:V_j\to V_k$ above by any other one $\gs$.
\end{proof}

There a natural description of the various classes of $G$-frames in \cite{W18}
(and their generalisation to $(G,\ga)$-frames) in terms of the 
Fourier coefficients, e.g.,

\begin{itemize}
\item {\it Irreducible $(G,\ga)$-frames}: There is only one nonzero
Fourier coefficent, which is a rank one orthogonal projection 
(up to a scalar multiple).
\item {\it Homogeneous $(G,\ga)$-frames}: There is only one nonzero
Fourier coefficent. \\
Equivalently, the Gramian is a $\rho$-matrix for some irreducible $\rho\in R$.
\item {\it Central $(G,\ga)$-frames}: All the Fourier coefficients are
$0$ or a scalar multiple of  $I$.
\end{itemize}

\begin{example} 
\label{CentralGramianex}
A $(G,\ga)$-frame with Gramian
$M_\ga(f)$ is 
{\bf central} if $f$ is a $\ga$-class function.
Since the irreducible $\ga$-characters are a basis for the $\ga$-class
functions, $f$ must have the form 
$$ f(g) = \sum_\rho a_\rho \chi_\rho(g)
= \sum_\rho a_\rho \trace(\rho(g)I), $$
and so, by (\ref{FTirredA}),
the Fourier coefficients of the frame are
$ (F_\ga f)_\rho = {|G|\over d_\rho} a_\rho I.$
For these to be orthogonal projections, we must have
$$a_\rho = {d_\rho\over|G|}={\chi_\rho(1)\over\ga(1,1)|G|}, $$
which, by taking $a_\rho=0$ or the above value, 
gives the characterisation of \cite{VW08} and \cite{CH18} for
normalised tight central $(G,\ga)$-frames in terms of their Gramian.
\end{example}

\vfil\eject

\begin{theorem}
\label{constructthm}
(Construction) Let $M_\ga(f)$ be an orthogonal projection,
i.e., the Gramian of some normalised tight
$(G,\ga)$-frame. 
Write its Fourier coefficients (which are rank $m_\xi$ orthogonal
projections) as
$$ (F_\ga f)_\xi = \sum_{j=1}^{m_\xi} w_{\xi,j}w_{\xi,j}^*,
\qquad w_\xi=(w_{\xi,1},\ldots,w_{\xi,m_\xi})\in (\CC^{d_\xi})^{m_\xi},
$$
where $\inpro{w_{\xi,j},w_{\xi,k}}=\gd_{jk}$. 
Let 
\begin{equation}
\label{vdefnconstruction}
v:=(\hbox{$\sqrt{d_\xi\over|G|}$}w_\xi)_{\xi\in R}\in V
:=\bigoplus_{\xi\in R} (\CC^{d_\xi})^{m_\xi}.
\end{equation}
Then $(\rho(g)v)_{g\in G}$ is a normalised tight $(G,\ga)$-frame for $V$ with
Gramian $M_\ga(f)$, where the unitary action $\rho:G\to\GL(V)$ is given by
$$ 
\rho(g)\bigl( (v_{\xi,1},\ldots,v_{\xi,m_\xi})_{\xi\in R}\bigr)
:=  (\xi(g)v_{\xi,1},\ldots,\xi(g)v_{\xi,m_\xi})_{\xi\in R}. $$
\end{theorem}

\begin{proof} 
Let $M_\ga(\nu)$ be the Gramian of $(\rho(g)v)_{g\in G}$, i.e.,
$$ \nu(g):=\inpro{\rho(g)v,v}
= \sum_{\xi\in R}\sum_{j=1}^{m_\xi} \inpro{\xi(g)w_{\xi,j},w_{\xi,j}}. $$
Then by the orthogonality of coordinates (\ref{coordorth}), 
and (\ref{FTirredGram}),
we calculate
\begin{align*}
(F_\ga\nu)_\eta 
&= \sum_{a\in G}\sum_{\xi\in R}{d_\xi\over|G|}\sum_{j=1}^{m_\xi}
\inpro{\xi(a)w_{\xi,j},w_{\xi,j}}\eta(a)^*
= \sum_{a\in G} {d_\xi\over|G|}\sum_{j=1}^{m_\eta}
\inpro{\eta(a)w_{\eta,j},w_{\eta,j}}\eta(a)^* \cr
&= \sum_{j=1}^{m_\eta} \sum_{a\in G} 
\trace( \eta(a) {d_\xi\over|G|} w_{\eta,j}w_{\eta,j}^*)\eta(a)^*
=\sum_{j=1}^{m_\eta} w_{\eta,j}w_{\eta,j}^*
=(F_\ga f)_\eta,
\end{align*}
as claimed.
\end{proof}

\begin{example}
Consider the normalised tight central $(G,\ga)$-frame with Gramian
$$ M_\ga\Bigl(\sum_{\xi\in S}{ d_\xi\over |G|}\chi_\xi\Bigr), 
\qquad d_\xi={\chi_\xi(1)\over\ga(1,1)}, $$
where $S\subset R$ (see Example \ref{CentralGramianex}).
Its nonzero Fourier coefficients are $I$ for $\xi\in S$.
By writing these as $I=\sum_j e_je_j^*$, we can realise this frame as 
$(\phi_g)_{g\in G}$, where
$$ \phi_g := \bigl( \hbox{$\sqrt{d_\xi\over|G|}$} 
\xi_{jk}(g)\bigr)_{\xi\in S, 1\le j,k\le d_\rho}.$$
This can be written compactly as $(\sqrt{d_\xi\over|G|}\xi)_{\xi\in S}
\in \bigoplus_{\xi\in S} \CC^{d_\xi\times d_\xi}$, with the
Frobenius inner product on $\CC^{d_\xi\times d_\xi}$.
\end{example}

A square matrix is the Gramian of a frame
(spanning sequence for a vector space) if and only if it is
positive semidefinite. Thus a $(G,\ga)$-matrix $M_\ga(f)$ is
the Gramian of a $(G,\ga)$-frame if and only if its Fourier 
coefficients are positive semidefinite.
Each such Fourier coefficient can be unitarily diagonalised,
giving 
$(F_\ga f)_\xi=\sum_j \gl_j w_{\xi,j}w_{\xi,j}^*,$
where the $w_{\xi,j}$ are orthonormal and $\gl_j>0$.
The frame can be realised as in Theorem \ref{constructthm},
where $w_\xi$ in (\ref{vdefnconstruction})
is replaced by
$(\sqrt{\gl_j}w_{\xi,j})_{1\le j\le m_\xi}$, 
and $m_\xi=\rank((F_\ga f)_\xi)$.


\section{Examples}

We now give some examples of $(G,\ga)$-matrices, their block 
diagonalisations (factorisation of the determinant), 
and Fourier decompositions.

\subsection{The Klein four-group}

The first group with a nontrivial Schur multipler is
the Klein four-group $G=\ZZ_2\times\ZZ_2$. 
We order this $1,a,b,ab=(0,0),(1,0),(0,1),(1,1)$, and write
$\nu(j,k)=\nu_{jk}$. 

For $\ga=1$, there are four
one-dimensional representations, giving
$$ M_1(\nu) = \pmat{ 
\nu_{00}&\nu_{10}&\nu_{01}&\nu_{11} \cr
\nu_{10}&\nu_{00}&\nu_{11}&\nu_{01} \cr
\nu_{01}&\nu_{11}&\nu_{00}&\nu_{10} \cr
\nu_{11}&\nu_{01}&\nu_{10}&\nu_{00} },
\qquad
E={1\over 2}\pmat{ 1&1&1&1 \cr 1&-1&1&-1 \cr 1&1&-1&-1 \cr 1&-1&-1&1},
 $$
with $\overline{E}^* M_1(\nu)\overline{E}$ being diagonal,
and
$$ \det(M_1(\nu)) = (\nu_{00}+\nu_{10}+\nu_{01}+\nu_{11})
(\nu_{00}-\nu_{10}+\nu_{01}-\nu_{11})
(\nu_{00}+\nu_{10}-\nu_{01}-\nu_{11})
(\nu_{00}-\nu_{10}-\nu_{01}+\nu_{11}). $$

For the nontrivial multiplier, we have $(G,\ga)$-matrices
$$ M_\ga(\nu) = \pmat{ 
\nu_{00}&\nu_{10}&\nu_{01}&\nu_{11} \cr
\nu_{10}&\nu_{00}&\nu_{11}&\nu_{01} \cr
\nu_{01}&-\nu_{11}&\nu_{00}&-\nu_{10} \cr
-\nu_{11}&\nu_{01}&-\nu_{10}&\nu_{00} },\quad
\ga:=\pmat{ 1&1&1&1\cr 1&1&1&1\cr 1&-1&1&-1\cr 1&-1&1&-1},  $$
and a single two-dimensional projective representation $\rho$ for $\ga$.
This representation, and a $\tilde\rho$ equivalent 
to it, are given by
$$ \rho(1)=\pmat{1&0\cr0&1},\quad
 \rho(a)=\pmat{0&1\cr1&0},\quad
 \rho(b)=\pmat{1&0\cr0&-1},\quad
 \rho(ab)=\pmat{0&-1\cr1&0}, $$
$$ \tilde\rho(1)=\pmat{1&0\cr0&1},\quad
 \tilde\rho(a)=\pmat{1&0\cr0&-1},\quad
 \tilde\rho(b)=\pmat{0&1\cr1&0},\quad
 \tilde\rho(ab)=\pmat{0&1\cr-1&0}. $$
where $\tilde\rho=T\rho T^{-1}$, $T={1\over\sqrt{2}}\pmat{1&1\cr1&-1}$.
For these, we have
$$ E_{\rho} = {1\over\sqrt{2}}\pmat{1&0&0&1\cr0&1&1&0\cr1&0&0&-1\cr0&-1&1&0}, 
\qquad
E_{\tilde\rho} 
= {1\over\sqrt{2}}\pmat{1&0&0&1\cr1&0&0&-1\cr0&1&1&0\cr0&1&-1&0}.$$
This shows that invariant subspace orthogonal decompositions 
$$ 
U_{\rho,\ga} = \bigoplus_{j=1}^{d_\rho} U_{\rho,\ga,j}, \qquad
V_{\rho,\ga} = \bigoplus_{j=1}^{d_\rho} V_{\rho,\ga,j}
 $$
are not unique.
\vfil\eject

\noindent
Here
$$ \overline{E_\rho}^* M_\ga(\nu)\overline{E_\rho}
=\pmat{ C_\rho &0\cr 0 & C_\rho}, \qquad C_\rho=(F_\ga\nu)_\rho^T
=\pmat{\nu_{00}+\nu_{01}&\nu_{10}-\nu_{11}\cr \nu_{10}+\nu_{11}&\nu_{00}-\nu_{01} } $$
$$ \overline{E_{\tilde\rho}}^* M_\ga(\nu)\overline{E_{\tilde\rho}}
=\pmat{ C_{\tilde\rho} &0\cr 0 & C_{\tilde\rho}}, 
\qquad C_{\tilde\rho}=(F_\ga\nu)_{\tilde\rho}^T
=\pmat{ \nu_{00}+\nu_{10}&\nu_{01}+\nu_{11} \cr
\nu_{01}-\nu_{11}&\nu_{00}-\nu_{10}}, $$
and
$$ \det(M_\ga(\nu)) = (\nu_{00}^2 +\nu_{11}^2
-\nu_{10}^2-\nu_{01}^2)^2. $$
We have a fine-scale Fourier decomposition of 
$M_\ga(\nu)$ into $(G,\ga)$-matrices
$$ M_\ga(\nu) = M_\ga(\nu_{\rho,1})+M_\ga(\nu_{\rho,2}), $$
where
$$ \nu_{\rho,1}={1\over2}(\nu_{00}+\nu_{01},\nu_{10}-\nu_{11},
\nu_{00}+\nu_{01},\nu_{11}-\nu_{10}), $$
$$ \nu_{\rho,2}={1\over2}(\nu_{00}-\nu_{01},\nu_{10}+\nu_{11},
\nu_{01}-\nu_{00},\nu_{10}+\nu_{11}). $$
%
The summands lie in the corresponding subalgebras of the $(G,\ga)$-matrices
$$ 
M_{\rho,\ga,1} =\Bigl\{ \pmat{ a&b&a&-b\cr b&a&-b&a\cr a&b&a&-b\cr b&a&-b&a}:a,b\in\CC\Bigr\}, \quad
M_{\rho,\ga,2} =\Bigl\{ \pmat{ c&d&-c&d\cr d&c&d&-c\cr -c&-d&c&-d\cr -d&-c&-d&c}:c,d\in\CC\Bigr\}, $$
for which every nonzero matrix has rank $2$.
By Proposition \ref{Malpharankprop}, we have
that there are no $(G,\ga)$-matrices of rank $1$ or $3$.

\subsection{The dihedral groups}
\label{dihedralgroupsect}

The next groups with nontrivial Schur multiplier are those of
order $8$, of which $\ZZ_2\times\ZZ_4$ and $D_8$ have Schur 
multiplier of order $2$, and $\ZZ_2^3$ which has Schur multiplier of
order $8$. We consider $G=D_8=\inpro{a,b:a^4=1,b^2=1,bab=a^{-1}}$.
In \cite{CH18}, it is shown that
for the nontrivial cocycle $\ga$ given by
$$ \ga( a^jb^k,a^\ell b^m) := i^{k\ell}, $$
there are inequivalent $2$-dimensional projective representations 
$\rho_1$ and $\rho_2$ 
for $\ga$ given by
$$ \rho_r (a^j b^k ) := \pmat{i^r&0\cr0& i^{1-r}}^j\pmat{0&1\cr1&0}^k. $$
We use the ordering $1,a,a^2,a^3,b,ab,a^2b,a^3b$ for $G$, so that
$$ M_\ga(\nu)= \pmat{ 
\nu_{1}&\nu_{a}&\nu_{a^2}&\nu_{a^3}&\nu_{b}&\nu_{ab}&\nu_{a^2b}&\nu_{a^3b} \cr
\nu_{a^3}&\nu_{1}&\nu_{a}&\nu_{a^2}&\nu_{a^3b}&\nu_{b}&\nu_{ab}&\nu_{a^2b} \cr
\nu_{a^2}&\nu_{a^3}&\nu_{1}&\nu_{a}&\nu_{a^2b}&\nu_{a^3b}&\nu_{b}&\nu_{ab} \cr
\nu_{a}&\nu_{a^2}&\nu_{a^3}&\nu_{1}&\nu_{ab}&\nu_{a^2b}&\nu_{a^3b}&\nu_{b} \cr
\nu_{b}&i\nu_{a^3b}&-\nu_{a^2b}&-i\nu_{ab}&\nu_{1}&i\nu_{a^3}&-\nu_{a^2}&-i\nu_{a} \cr
-i\nu_{ab}&\nu_{b}&i\nu_{a^3b}&-\nu_{a^2b}&-i\nu_{a}&\nu_{1}&i\nu_{a^3}&-\nu_{a^2} \cr
-\nu_{a^2b}&-i\nu_{ab}&\nu_{b}&i\nu_{a^3b}&-\nu_{a^2}&-i\nu_{a}&\nu_{1}&i\nu_{a^3} \cr
i\nu_{a^3b}&-\nu_{a^2b}&-i\nu_{ab}&\nu_{b}&i\nu_{a^3}&-\nu_{a^2}&-i\nu_{a}&\nu_{1} 
}, $$
and, e.g.,
$$ \rho_1=\bigl( 
\pmat{1&0\cr0&1},\pmat{i&0\cr0&1}, \pmat{-1&0\cr0&1},\pmat{-i&0\cr0&1},
\pmat{0&1\cr1&0},\pmat{0&i\cr1&0},\pmat{0&-1\cr1&0},\pmat{0&-i\cr1&0}
\bigr). $$
Here
\begin{equation}
\label{Eexamp}
E:={1\over2}\pmat{
1&0&0&1&1&0&0&1 \cr
i&0&0&1&-1&0&0&-i \cr
-1&0&0&1&1&0&0&-1 \cr
-i&0&0&1&-1&0&0&i \cr
0&1&1&0&0&1&1&0 \cr
0&i&1&0&0&-1&-i&0 \cr
0&-1&1&0&0&1&-1&0 \cr
0&-i&1&0&0&-1&i&0 },
\end{equation}
and $\overline{E}^*M_\ga(\nu)\overline{E}$ is block diagonal,
with $2\times 2$ diagonal blocks
$(F_\ga \nu)_{\rho_1}^T,(F_\ga \nu)_{\rho_1}^T,
(F_\ga \nu)_{\rho_2}^T,(F_\ga \nu)_{\rho_2}^T$, where
$$ (F_\ga \nu)_{\rho_1} = \pmat{
 \nu_1-i\nu_a-\nu_{a^2}+i\nu_{a^3} & \nu_b+\nu_{ab}+\nu_{a^2b}+\nu_{a^3b} \cr
\nu_b-i\nu_{ab}-\nu_{a^2b}+i\nu_{a^3b}  & \nu_1+\nu_a+\nu_{a^2}+\nu_{a^3}  
}, $$
$$ (F_\ga \nu)_{\rho_2} = \pmat{
 \nu_1-\nu_a+\nu_{a^2}-\nu_{a^3} & \nu_b+i\nu_{ab}-\nu_{a^2b}-i\nu_{a^3b} \cr
\nu_b-\nu_{ab}+\nu_{a^2b}-\nu_{a^3b}  & \nu_1+i\nu_a-\nu_{a^2}-i\nu_{a^3}  
}. $$
Thus $\det(M_\ga(\nu))$ factors 
as $\det((F_\ga \nu)_{\rho_1})^2\det((F_\ga \nu)_{\rho_2})^2$. 
It also happens that ${E}^*M_\ga(\nu){E}$ is block diagonal,
i.e., the $U_{\rho,\ga,j}$ are invariant subspaces of the the
$(G,\ga)$-matrices. This is because conjugation permutes the subspaces
$U_{\rho,\ga,j}$, i.e.,
from (\ref{Eexamp}) it is apparent that
\begin{equation}
\label{UrhoalphaVrhoalpharelat}
V_{\rho_1,\ga,1} 
= U_{\rho_2,\ga,2}, \quad
V_{\rho_1,\ga,2} 
= U_{\rho_1,\ga,2}, \quad
V_{\rho_2,\ga,1} 
= U_{\rho_2,\ga,1}, \quad
V_{\rho_2,\ga,2} 
= U_{\rho_1,\ga,1}. 
\end{equation}
The diagonal blocks of ${E}^*M_\ga(\nu){E}$ are
$$\pmat{
\nu_1+i\nu_a-\nu_{a^2}-i\nu_{a^3}&\nu_b+i\nu_{ab}-\nu_{a^2b}-i\nu_{a^3b}\cr
\nu_b-\nu_{ab}+\nu_{a^2b}-\nu_{a^3b}&\nu_1-\nu_a+\nu_{a^2}-\nu_{a^3}},
\pmat{
\nu_1-i\nu_a-\nu_{a^2}+i\nu_{a^3}&\nu_b-i\nu_{ab}-\nu_{a^2b}+i\nu_{a^3b}\cr
\nu_b+\nu_{ab}+\nu_{a^2b}+\nu_{a^3b}&\nu_1+\nu_a+\nu_{a^2}+\nu_{a^3}}, $$
$$\pmat{
\nu_1-\nu_a+\nu_{a^2}-\nu_{a^3}&\nu_b-\nu_{ab}+\nu_{a^2b}-\nu_{a^3b} \cr
\nu_b+i\nu_{ab}-\nu_{a^2b}-i\nu_{a^3b}&\nu_1+i\nu_a-\nu_{a^2}-i\nu_{a^3} },
\pmat{
\nu_1+\nu_a+\nu_{a^2}+\nu_{a^3}&\nu_b+\nu_{ab}+\nu_{a^2b}+\nu_{a^3b}\cr
\nu_b-i\nu_{ab}-\nu_{a^2b}+i\nu_{a^3b}&\nu_1-i\nu_a-\nu_{a^2}+i\nu_{a^3} }. $$
We observe that these are all different.

\section{Other $(G,\ga)$-matrices}
\label{otherGgamats}

The definition (\ref{Malphanuformula}) for $M_\ga(\nu)$ and 
that of the Fourier transform $F_\ga$ are motivated by our 
analysis of projective group frames. Since these notions are so new,
we now provide the tools to compare the variants. As the theory 
evolves, perhaps it will become apparent if there are ones which are best.

In (\ref{projGramian}),
we define the Gramian of $(gv)_{g\in G}$ so that it factors
$V^*V$, where $V=[gv]_{g\in G}$. 
In \cite{CH18} the transpose (or equivalently the complex conjugate)
of this is considered (for unitary representations), 
i.e., the matrix with $(g,h)$-entry
$$ \inpro{\phi_g,\phi_h} 
=\overline{\inpro{\phi_h,\phi_g}}
= {\inpro{v,\rho(g^{-1}h)v}\over\overline{\ga}(g,g^{-1}h)}
= \ga(g,g^{-1}h) \inpro{v,\rho(g^{-1}h)v}. $$
For $\ga$ unitary or not, we say that $A\in\CC^{G\times G}$ a 
{\bf $[G,\ga]$-matrix} if it has this form, i.e.,
\begin{equation}
\label{[G,alpha]matdef}
a_{g,h} = M(\nu)_{g,h} 
:= \ga(g,g^{-1}h) \nu(g^{-1}h) , \qquad \nu\in\CC^G. 
\end{equation}
In \cite{CH18}, the formula (\ref{[G,alpha]matdef}) 
is written as
\begin{equation}
\label{Mnualtdef}
M(\nu)_{g,h} =\ga(g,g^{-1}h)\nu(g^{-1}h)
= { \ga(g,g^{-1})\ga(1,1) \over \ga(g^{-1},h) } \nu(g^{-1}h). 
\end{equation}
We observe that $1/\ga$ is cocycle, and that a $[G,\ga]$-matrix
is a $(G,1/\ga)$-matrix, i.e.,
\begin{equation}
\label{Malphainv}
M(\nu) = M_{1/\ga}(\nu). 
\end{equation}
Moreover, for $\ga$ unitary, the complex conjugate of a $(G,\ga)$-matrix is a $[G,\ga]$-matrix, i.e.,
\begin{equation}
\label{Gmatcompconj}
\overline{M_\ga(\nu)} = M(\overline{\nu}).
\end{equation}

\begin{proposition} 
The $(G,\ga)$-matrices and $[G,\ga]$-matrices are 
the transposes of each other, i.e.,
$$ M_\ga(\nu)^T= M(\mu), \qquad
\mu(g):={\nu(g^{-1})\over\ga(1,1)\ga(g,g^{-1})}, $$
$$ M(\nu)^T= M_\ga(\mu), \qquad
\mu(g):= \ga(1,1)\ga(g,g^{-1}) \nu(g^{-1}). $$
\end{proposition}

\begin{proof} 
Since $\ga(g,g^{-1}h)\ga(h,h^{-1}g) = \ga(g,1)\ga(g^{-1}h,h^{-1}g),$
we calculate
$$ (M_\ga(\nu)^T)_{g,h} = M_\ga(\nu)_{h,g}
= {\nu(h^{-1}g)\over\ga(h,h^{-1}g)} 
= \ga(g,g^{-1}h){\nu((g^{-1}h)^{-1})\over\ga(1,1)\ga(g^{-1}h,(g^{-1}h)^{-1})}.
$$
The other follows similarly, or by a change of variables.
\end{proof}

\begin{example}
Taking the transpose of the diagonalisation of
Theorem \ref{Galphamatdiag} gives
$$ E^* M_\ga(\nu)^T E
= \overline{E}^T M_\ga(\nu)^T (\overline{E}^*)^T
= \diag\bigl( (F_\ga \nu)_\rho : \rho\in R,1\le k\le d_\rho\bigr), $$
so that the $U_{\rho,\ga,j}$ are invariant subspaces of $M_\ga(\nu)^T$,
and hence of the $[G,\ga]$-matrices.
\end{example}

Group matrices can also be defined to have $(g,h)$-entries
of the form $\nu(gh^{-1})$ \cite{J18}. To transform these to matrices of the
above type, we consider the unitary involution $J$ given by
$Je_h:= e_{h^{-1}}$, i.e., $(J)_{g,h}=\gd_{g,h^{-1}}$. Then
$$ (JAJ)_{g,h} = a_{g^{-1},h^{-1}}, $$
and so for $A=M_\ga(\nu)$ and $A=M(\nu)$, we have
$$ (JM_\ga(\nu)J)_{g,h} = {\nu(gh^{-1})\over\ga(g^{-1},gh^{-1})}, \qquad
(JM(\nu)J)_{g,h} 
=\ga(g^{-1},gh^{-1})\nu(gh^{-1}),
$$
which would provide the natural definitions for $(G,\ga)$-matrices
of this type.

\begin{example}
For ordinary representations, 
Theorem \ref{ordinaryblockdiag} gives
$$ (JE)^* (JM_\ga(\nu)J) (JE) 
= \diag\bigl((F_\ga\nu)_\rho:\rho\in R,1\le k\le d_\rho\bigr), $$
so the ``group matrices'' $JM_\ga(\nu)J$ are block diagonalised by $JE$, which 
has $(\rho,k)$-blocks
$$
\hbox{$\sqrt{d_\rho\over|G|}$} 
[J\rho_{k1},J\rho_{k2},\ldots,J\rho_{kd_\xi}]
=\hbox{$\sqrt{d_\rho\over|G|}$} 
[\overline{\rho_{1k}},\overline{\rho_{2k}},\ldots,\overline{\rho_{d_\xi k}}], $$
since 
$$(J\rho_{jk})_g = \rho_{jk}(g^{-1})
=(\rho(g^{-1}))_{jk}
=(\rho(g)^*)_{jk}
=\overline{\rho_{kj}(g)}
. $$
This is the Theorem 61 of \cite{J18}. 
\end{example}


\bibliographystyle{alpha}
\bibliography{references}
\nocite{*}



\vfil

\end{document}

\end{document}